\documentclass[12pt,reqno]{amsart}
\usepackage{amsmath}
\usepackage{amssymb}
\usepackage{amstext}
\usepackage{a4wide}
\usepackage{graphicx}
\allowdisplaybreaks \numberwithin{equation}{section}
\usepackage{color}

\numberwithin{equation}{section}

\newtheorem{theorem}{Theorem}[section]
\newtheorem{proposition}[theorem]{Proposition}
\newtheorem{corollary}[theorem]{Corollary}
\newtheorem{lemma}[theorem]{Lemma}

\theoremstyle{definition}

\newtheorem{definition}[theorem]{Definition}

\theoremstyle{remark}
\newtheorem{remark}[theorem]{Remark}

\begin{document}

\title
{Desingularization of Vortex Rings in 3 dimensional Euler Flows}

 \author{Daomin Cao, Jie Wan,  Weicheng Zhan}

\address{Institute of Applied Mathematics, Chinese Academy of Sciences, Beijing 100190, and University of Chinese Academy of Sciences, Beijing 100049,  P.R. China}
\email{dmcao@amt.ac.cn}
\address{Institute of Applied Mathematics, Chinese Academy of Sciences, Beijing 100190, and University of Chinese Academy of Sciences, Beijing 100049,  P.R. China}
\email{wanjie15@mails.ucas.edu.cn}
\address{Institute of Applied Mathematics, Chinese Academy of Sciences, Beijing 100190, and University of Chinese Academy of Sciences, Beijing 100049,  P.R. China}
\email{zhanweicheng16@mails.ucas.ac.cn}


\begin{abstract}
In this paper, we are concerned with nonlinear desingularization of steady vortex rings of three-dimensional incompressible Euler fluids. We focus on the case when the vorticity function has a simple discontinuity, which corresponding to a jump in vorticity at the boundary of the cross-section of the vortex ring. Using the vorticity method, we construct a family of steady vortex rings which constitute a desingularization of the classical circular vortex filament in several kinds of domains. The precise localization of the asymptotic singular vortex filament is proved to depend on the circulation and the velocity at far fields of the vortex ring. Some qualitative and asymptotic properties are also established. Comparing with known results, our work actually enriches and advances the study on this problem.
\end{abstract}

\maketitle

\section{Introduction}
The motion of an incompressible steady Euler fluid in $\mathbb{R}^3$ is governed by the following Euler equations

\begin{equation}\label{1-1}
(\mathbf{v}\cdot\nabla)\mathbf{v}=-\nabla P,
\end{equation}
\begin{equation}\label{1-2}
 \nabla\cdot\mathbf{v}=0,
\end{equation}
where $\mathbf{v}=[v_1,v_2,v_3]$ is the velocity field and $P$ is the scalar pressure.

In this paper, we are concerned with desingularization of steady vortex rings of axisymmetric incompressible Euler system without swirl in several types of simple connected domains. Since the flow should be axisymmetric without swirl, the conservation of mass equation $\eqref{1-2}$ is equivalent to the existence of the Stokes stream function $\psi$ in cylindrical coordinates $(r,\theta,z)$, satisfying
\begin{equation}\label{1-3}
  \mathbf{v}(r,\theta,z)=\frac{1}{r}\Big{(}-\frac{\partial\psi}{\partial z}\mathbf{e}_r+\frac{\partial\psi}{\partial r}\mathbf{e}_z\Big{)}
\end{equation}
and the associated vorticity $\boldsymbol{\omega}:=\text{curl}\mathbf{v}$ is given by
\begin{equation*}
  \boldsymbol{\omega}(r,\theta,z)=-\Big(\frac{\partial}{\partial r}\Big(\frac{1}{r}\frac{\partial\psi}{\partial r}\Big)+\frac{\partial}{\partial z}\Big(\frac{1}{r}\frac{\partial\psi}{\partial z}\Big)\Big)\mathbf{e}_\theta:=\omega^\theta(r,\theta,z)\mathbf{e}_\theta,
\end{equation*}
where $\{\mathbf{e}_r, \mathbf{e}_\theta, \mathbf{e}_z\}$ is the usual cylindrical coordinate frame.
Note that the conservation of momentum equation $\eqref{1-1}$ can be rewritten as
\begin{equation*}
  \boldsymbol{\omega}\times\mathbf{v}=-\nabla\Big(P+\frac{|\mathbf{v}|^2}{2}\Big).
\end{equation*}
If $\boldsymbol{\omega}=rf(\psi)\mathbf{e}_\theta$ for some vorticity function $f:\mathbb{R}\to \mathbb{R}$ and $F'=f$, then
\begin{equation*}
  \boldsymbol{\omega}\times\mathbf{v}=-\nabla(F(\psi)).
\end{equation*}
Therefore the problem is thus reduced to the following semilinear elliptic problem
\begin{equation}\label{1-4}
\mathcal{L}\psi:=-\frac{1}{r}\frac{\partial}{\partial r}\Big(\frac{1}{r}\frac{\partial\psi}{\partial r}\Big)-\frac{1}{r^2}\frac{\partial^2\psi}{\partial z^2}=f(\psi).
\end{equation}
Once we find the Stokes stream function $\psi$, the velocity of the flow is given by $\eqref{1-3}$ and the pressure is given by $P=F(\psi)-\frac{1}{2}|\mathbf{v}|^2$.

Since $\nabla\times((\mathbf{v}\cdot\nabla)\mathbf{v})=0$, if we let $\zeta=\omega^\theta/r$, then $\eqref{1-1}$ reduces to simply
\begin{equation}\label{1-5}
  \mathbf{v}\cdot\nabla\zeta=0.
\end{equation}
Recalling $\eqref{1-3}$, $\eqref{1-5}$ can be written as
\begin{equation}\label{1-6}
  \frac{\partial(\psi,\zeta)}{\partial(r,z)}=0,
\end{equation}
with $\partial(\cdot,\cdot)/\partial(r,z)$ the determinant of the gradient matrix.

The motion of vortex rings has been investigated since the work of Helmholtz \cite{He} in 1858 and Kelvin \cite{Tho} in 1867. In \cite{Hi}, Hill constructed a classical example of steady vortex rings(called Hill's spherical vortex) whose support is a ball. Kelvin and Hicks showed that if the vortex ring with circulation $\kappa$ has radius $r_*$ and its cross-section $\varepsilon$ is small, then the vortex ring moves at the velocity(see \cite{Lamb, Tho})
\begin{equation}\label{1-8}
  \frac{\kappa}{4\pi r_*}\Big(\log \frac{8r_*}{\varepsilon}-\frac{1}{4}\Big).
\end{equation}
Fraenkel first proved that one can construct flows such that its vorticity is supported in an arbitrarily small toroidal region (see \cite{Fr3, Fr1, Fr2}).
More precisely, he proved that for small $\varepsilon>0$, there exists a steady vortex ring whose vortex cross-section is of the order of $\varepsilon$ and whose velocities satisfy asymptotically $\eqref{1-8}$. For a detailed and historical description of this problem, we refer to \cite{DV, BF1}. \cite{VAT} is a good historical overview on the development of vortex dynamics.

Roughly speaking, there are two methods to investigate the problem of steady vortex rings. The first one is called the stream-function method, namely, finding a solution of $\eqref{1-4}$ with the desired properties, see \cite{AS, BF2, DV, Fr1, Fr2, BF1, Ni, Yang2} and reference therein. By using the stream-function method, Fraenkel and Berger \cite{BF1} constructed solutions of $\eqref{1-4}$ in the whole space with prescribed constant velocity at far fields. Nonlinear desingularization for general free-boundary problems was studied in \cite{BF2}, but asymptotic behaviour of the solutions they constructed could not be studied precisely because of the presence of a Lagrange multiplier in the nonlinearity $f$.  In \cite{Ta}, Tadie studied the asymptotic behaviour by letting the flux diverge. More steady vortex rings can also be obtained by using the mountain pass theorem proposed by Ambrosetti and Rabinowitz \cite{AR}(see \cite{AS, AM, Ni} for example). Yang studied the asymptotic behaviour of a family of solutions $\psi_\varepsilon$ of $\eqref{1-4}$ with $f_\varepsilon(t)=\frac{1}{\varepsilon^2}h(t)$ for some smooth function $h$ \cite{Yang2}. However, their limiting objects are degenerate vortex rings with vanishing circulation. Recently, de Valeriola and Van Schaftingen obtained some desingularization results of steady vortex rings by using the stream-function method \cite{DV}. They proved that, for given $W>0$ , $\kappa>0$ and $f_\varepsilon(t)=\frac{1}{\varepsilon^2}(t)_+^p(p>1)$, there exists a family of steady vortex rings of small cross-section with the circulation $\kappa_\varepsilon \to \kappa $ and the velocity satisfies  $\mathbf{v}_\varepsilon \to -W\log \frac{1}{\varepsilon}\mathbf{e}_z$ at infinity as $\varepsilon \to 0$. Moreover, the steady vortex rings will concentrate at a circular vortex filament. Several types of domains were considered in \cite{DV}.

Another method to study nonlinear desingularization of vortex rings is called the vorticity method, which solves variational problem for the potential vorticity $\zeta$ (see \cite{BB, Be, B1, B2, Dou2, FT}). In contrast with the stream-function method, the vorticity method has strong physical motivation. In \cite{Be}, Benjamin proposed a variational principle for the vorticity. The idea was to seek extremals of the energy relative to the set of rearrangements of a fixed function. In \cite{FT}, Friedman and Turkington proved desingularization results of vortex rings in the whole space when the vorticity function $f$ is a step function. They located the vortex rings by constraining the impulse of the flow to be a constant. Because of this, the velocities of the flows at far fields became Lagrange multipliers and hence were undetermined. Following Benjamin's idea, Burton et al. investigated the existence of vortex rings in various cases (see \cite{BB, B1, B2, Dou2}). We should mention that the approach they adopted is different in some important aspects from the one Benjamin envisaged. Very recently, Dekeyser used the vorticity method to study desingularzation of a steady vortex pair in the lake equations of which the three-dimensional axisymmetric Euler equations are a particular case (see \cite{D1,D2}). Specifically, he constructed a family of steady solutions of the lake model which were proved to converge to a singular vortex pair. The precise localization of the asymptotic singular vortex pair depends on the depth function and the Coriolis parameter. Note that the lake domains therein were not necessarily regular.

In the present paper, we are interested in the case when the vorticity function $f(t)$ is a step function, which has a simple discontinuity at $t=0$. This simplest of all admissible vorticity distributions has been a favourite for over a century. It is also the vorticity of the Prandtl-Batchelor theorem about the inviscid limit of flows with closed streamlines (refer to \cite{Fr2}). However, the discontinuity of $f$ poses some challenging problems in analysis in the study of $\eqref{1-4}$. For the case of the whole space, as mentioned above, Friedman and Turkington \cite{FT} obtained some results on the desingularization. However, the method we adopt here is quite different from theirs. In \cite{D2}, Dekeyser studied the asymptotic behavior of shrinking vortex pairs in the lake equations in bounded domains. In the case of vortex ring, our last result (see Theorem $\ref{thm4}$ below) actually improves his result to some extent.  To the best of our knowledge, there are no other results in this aspect. In this paper, we mainly use the vorticity method to study desingularization of steady vortex rings in several kinds of domains, namely, smooth bounded domains, infinite pipe, the whole space and
exterior domain in $\mathbb{R}^3$. We adopt Burton's method to show the existence of vortex rings in brief. It is instructive to compare the previous solutions mentioned above with ours. For this aspect, one can refer to \cite{BB} for detailed description. Our focus is the asymptotic behaviour of those solutions. We note that in this case the method used in \cite{DV} seems cannot be applied, since one need some continuity assumptions about the vorticity function $f$ to ensure that the functional is Gateaux differentiable and the critical point theory can be used. Our strategy is to analyze the Green's function carefully and estimate the order of energy as optimally as possible. The key point is that in order to maximize the energy, those solutions have to be concentrated. Our method is inspired by the works of \cite{DV, Fr1, FT, Tur83, Tur89}. It is worth noting that our method does not require the connectness of the vortex core. Note that in \cite{BF2, DV, FT, Ta, Yang2}, the crucial estimates for the diameter of the cross-section via the vortex strength parameter depend on the connectness of the vortex core. We also remark that our method can also be applied to more general domains.

We note that there is a similar situation with similar results in the study of vortex pairs for the two-dimensional Euler equation(see, for example, \cite{CLW, Cao1, Cao2, LYY, SV, Tur83}). Finally, what is worth mentioning is that the uniqueness of the vortex rings remains open. Amick and Fraenkel proved the uniqueness of Hill's spherical vortex in \cite{AF}. Without the uniqueness, one cannot verify whether the solutions constructed by the vorticity method or the stream-function method are the same. Several results can be found, see \cite{AF, AF2, Cao2, Hi}.

The paper is organized as follows. In section 2, we state the main results and give some remarks. In section 3, we study vortex rings in an infinite pipe and the whole space since these two situations are similar. In section 4, we investigate vortex rings outside a ball which is a little different from other cases. In section 5, we consider the case of smooth bounded domains in brief.

\section{Main results}
Throughout the sequel we shall use the following notations: $x=(r,\theta,z)$ denotes the cylindrical coordinates of $x\in\mathbb{R}^3$; $\{\mathbf{e}_r, \mathbf{e}_\theta, \mathbf{e}_z\}$ represents the associated standard orthonormal frame; $\Pi=\{(r,z)~|~r>0, z\in\mathbb{R}\}$ denotes a meridional half-plane($\theta$=constant); Lebesgue measure on $\mathbb{R}^N$ is denoted $\textit{m}_N$, and is to be understood as the measure defining any $L^p$ space and $W^{1,p}$ space, except when stated otherwise; $\nu$ denotes the measure on $\Pi$ having density $2\pi r$ with respect to $\textit{m}_2$, $|\cdot|$ denotes the $\nu$ measure; $B_\delta(y)$ denotes the open ball in $\Pi$ of radius $\delta$ centered at $y$; $I_A$ denotes the characteristic function of $A\subseteq\Pi$.

Let $U\subseteq \mathbb{R}^3$ be a domain with cylindrically symmetric about the $z$ axis. Let $D=U\cap\Pi$.
\begin{definition}
The set $D$ is admissible if it is one of the following four types
\begin{itemize}
  \item[(a)]  $\{(r,z)\in \Pi~|~0<r<d, z\in\mathbb{R}\}\ \text{for some}\  d \in \mathbb{R}_+$,
  \item[(b)]  $\Pi$,
  \item[(c)]  $\{(r,z)\in \Pi ~|~ r^2+z^2>d^2\}\ \text{for some}\  d \in \mathbb{R}_+$,
  \item[(d)]  $\{(r,z)\in \Pi~|~r^2+z^2<b^2\}$ or $(0,b)\times(-c,c), \text{for some}\  b,c \in \mathbb{R}_+$,
\end{itemize}

\end{definition}

\begin{definition}
  Let D be admissible. The Hilbert space $H(D)$ is the completion of $C_0^\infty(D)$ with the scalar products
\begin{equation*}
  \langle u,v\rangle_H=\int_D\frac{1}{r^2}\nabla u\cdot\nabla v d\nu.
\end{equation*}
We define inverses $K$ for $\mathcal{L}$ in the weak solution sense,
\begin{equation}\label{2-1}
  \langle Ku,v\rangle_H=\int_Duv d\nu \ \ \ for\  all\  v\in H(D), \ \ when \ u \in L^{10/7}(D,r^3drdz).
\end{equation}
\end{definition}
Note that we can construct $Ku\in H(D)$ by the Riesz representation theorem. It is not hard to check that $K$ is well-defined. See \cite{BB,B1} for instance.

Recall that $\eqref{1-1}$ can be written as
\begin{equation}\label{1-7}
 \frac{\partial(\psi,\zeta)}{\partial(r,z)}=0.
\end{equation}
\begin{definition}
  $(\psi,\zeta)\in C^{1}(D)\times L^\infty(D)$ is called a weak solution of $\eqref{1-7}$ if for any $\varphi\in C_0^\infty(D)$,
\begin{equation}\label{2-2}
  \int_D\zeta~ \nabla^{\bot}\psi\cdot\nabla\varphi drdz=0,
\end{equation}
 where $\nabla^\bot=(\partial_z,-\partial_r)$.
\end{definition}
Let $K(r,z,r',z')$ be the Green's function of $\mathcal{L}$ in $D$, with respect to zero Dirichlet data and measure $\nu$. It is not hard to show that the operator $K$ is an integral operator with kernel $K(r,z,r',z')$ for all cases considered in this paper. We shall use this Green's representation formula directly without further explanation.

Let $\mathcal{C}_r=\{x\in\mathbb{R}^3~|x_1^2+x_2^2=r^2,z=0\}$ be a circle of radius $r$ on the plane perpendicular to $\mathbf{e}_z$. For a set $A\subseteq \mathbb{R}^3$ axisymmetric around $\mathbf{e}_z$, we define the axisymmetric distance as follows
\begin{equation*}
  dist_{\mathcal{C}_r}(A)=\sup_{x\in A}\inf_{x'\in{\mathcal{C}_r}}|x-x'|.
\end{equation*}

Our first result is  desingularization of vortex rings in an infinite pipe.
\begin{theorem}\label{thm1}
Let $U=\{(r,\theta,z)\in \mathbb{R}^3~|~0\leq r<d\}$ for some $d>0$ and let $D=U\cap\Pi$. Then for every $W>0$ and all sufficiently large $\lambda$, there exists a weak solution $(\psi_\lambda,\zeta_\lambda)$ of $\eqref{1-7}$  satisfying
\begin{itemize}
\item[(i)]For any $p>1$, $0<\alpha<1$, $\psi_\lambda\in W^{2,p}_{\text{loc}}(D)\cap C^{1,\alpha}(\bar{D})$ and satisfies
\begin{equation*}
  \mathcal{L}\psi_\lambda=\zeta_\lambda\ \ \text{a.e.} \ \text{in} \ D.
\end{equation*}

\item[(ii)] $(\psi_\lambda,\zeta_\lambda)$ is of the form
\begin{equation*}
  \begin{split}
    & \psi_\lambda=K\zeta_\lambda-\frac{W\log\lambda}{2}r^2-\mu_\lambda,\ \ \zeta_\lambda=\lambda I_{\Omega_\lambda}, \\
     & \Omega_\lambda=\{x\in D~|~\psi_\lambda(x)>0 \},\ \ \ \lambda|\Omega_\lambda|=1,
\end{split}
\end{equation*}
for some $\mu_\lambda>0$ depending on $\lambda$.

 \item[(iii)] For any $\alpha\in(0,1)$, there holds $$diam(\Omega_\lambda) \le 4d \lambda^{-\frac{\alpha}{2}}$$
 provided $\lambda$ is large enough.
 Moreover,
 \begin{equation*}
\begin{split}
   \lim_{\lambda\to +\infty} \frac{\log diam(\Omega_\lambda)}{\log (\lambda^{-\frac{1}{2}})} & =1, \\
   \lim_{\lambda \to +\infty}dist_{\mathcal{C}_{r_*}}(\Omega_\lambda)&=0,
\end{split}
\end{equation*}
where
\[
r_*=\left\{
   \begin{array}{lll}
        \frac{1}{16\pi^2 W} &    \text{if} & W>1/(16\pi^2d), \\
         d                  &    \text{if} & W\le 1/(16\pi^2d).
    \end{array}
   \right.
\]
Furthermore, if~$W>1/(16\pi^2 d)$, then there exists a constant $R_0>1$ independent of $\lambda$ such that
$$diam(\Omega_\lambda)\le R_0\lambda^{-\frac{1}{2}}$$
provided $\lambda$ is large enough; meanwhile, as $\lambda \to +\infty$,
 $$\mu_\lambda   = \Big(\frac{r_*}{8\pi^2}-\frac{Wr_*^2}{2}\Big)\log\lambda+O(1).$$
\item[(iv)] Let $$\mathbf{v}_\lambda=\frac{1}{r}\Big(-\frac{\partial\psi_\lambda}{\partial z}\mathbf{e}_r+\frac{\partial\psi_\lambda}{\partial r}\mathbf{e}_z\Big),$$ then
\begin{equation*}
\begin{split}
    & \mathbf{v_\lambda}\cdot\mathbf{n}=0 \  \ \text{on}\ \partial U,\\
    & \mathbf{v}_\lambda\to -W\log\lambda~\mathbf{e}_z\ \  \text{at}\ \infty, \  \  \text{as}\  \lambda \to +\infty,
\end{split}
\end{equation*}
where $\mathbf{n}$ is the unit outward normal of $\partial U$. Moreover, as $r\to 0$,
\begin{equation*}
     \frac{1}{r}\frac{\partial\psi_\lambda}{\partial z}\to 0\ \text{and}~ \ \frac{1}{r}\frac{\partial\psi_\lambda}{\partial r}\  \text{approaches a finite limit}.
\end{equation*}
\item[(v)]Let $a(\lambda)$ be any point of $\Omega_\lambda$, let $M>0$ be fixed. Then as $\lambda \to +\infty$,
\begin{equation}\label{309}
  K\zeta_\lambda(\cdot)-{K(\cdot,a(\lambda))} \to 0 \ \  \text{in}\ \  W^{1,p}(D_M),\ \ ~1\le p<2,
\end{equation}
and hence in $L^r(D_M)$, $1\le r<\infty$, where $D_M:=D\cap\{|z|\le M\}$. Moreover, for any $\alpha\in (0,1)$, as $\lambda \to +\infty$,
\begin{equation*}
  K\zeta_\lambda(\cdot)-{K(\cdot,a(\lambda))} \to 0 \ \  \text{in}\ \  C^{1,\alpha}_{loc}(D\backslash\{(r_*,0)\}).
\end{equation*}

\end{itemize}
\end{theorem}

\begin{remark}
  Burton has considered the question of existence of vortex rings in a cylinder \cite{B1}. In \cite{Dou2}, Douglas further investigated this question and generalized Burton's results. However they did not study the asymptotic behaviour.
\end{remark}

\begin{remark}
  Our result is similar to Theorem 2 of \cite{DV} where the vorticity function is continuous. Note that the velocity $-W\log\lambda$ of the vortex ring is less than predicted by the Kelvin-Hicks formula $\eqref{1-8}$ when $W<1/(16\pi^2d)$. This phenomenon also arises even when the vorticity function $f$ is smooth. Some possible explanations were given in \cite{DV}, we do not enter into details here.
\end{remark}

Similarly we can study desingularization of vortices in the whole space.

\begin{theorem}\label{thm2}
Let $D=\Pi$. Then for every $W>0$ and all sufficiently large $\lambda$, there exists a weak solution $(\psi_\lambda,\zeta_\lambda)$ of $\eqref{1-7}$ satisfying
\begin{itemize}
\item[(i)]For any $p>1$, $0<\alpha<1$, $\psi_\lambda\in W^{2,p}_{\text{loc}}(D)\cap C^{1,\alpha}_{\text{loc}}(\bar{D})$ and satisfies
\begin{equation*}
  \mathcal{L}\psi_\lambda=\zeta_\lambda\ \ \text{a.e.} \ \text{in} \ D.
\end{equation*}

\item[(ii)] $(\psi_\lambda,\zeta_\lambda)$ is of the form
\begin{equation*}
  \begin{split}
    & \psi_\lambda=K\zeta_\lambda-\frac{W\log\lambda}{2}r^2-\mu_\lambda,\ \ \zeta_\lambda=\lambda I_{\Omega_\lambda}, \\
     & \Omega_\lambda=\{x\in D~|~\psi_\lambda(x)>0 \},\ \ \ \lambda|\Omega_\lambda|=1,
\end{split}
\end{equation*}
for some $\mu_\lambda>0$ depending on $\lambda$.

 \item[(iii)]There exists a constant $R_0>1$ independent of $\lambda$ such that  $$diam(\Omega_\lambda) \le R_0 \lambda^{-\frac{1}{2}}$$
 provided $\lambda$ is large enough.
 Moreover, set $r_*=\frac{1}{16\pi^2W}$, then
 \begin{equation*}
\begin{split}
   \lim_{\lambda \to +\infty}dist_{\mathcal{C}_{r_*}}(\Omega_\lambda)=0,&\\
   \mu_\lambda             = \Big(\frac{r_*}{8\pi^2}-\frac{Wr_*^2}{2}\Big)\log\lambda+O(1)&.
\end{split}
\end{equation*}

\item[(iv)]As $\lambda \to +\infty$,
\begin{equation*}
  \mathbf{v}_\lambda=\frac{1}{r}\Big(-\frac{\partial\psi_\lambda}{\partial z}\mathbf{e}_r+\frac{\partial\psi_\lambda}{\partial r}\mathbf{e}_z\Big)\to -W\log\lambda~\mathbf{e}_z\ \ \text{at} \ \infty.
\end{equation*}
 Moreover, as $r\to 0$,
\begin{equation*}
     \frac{1}{r}\frac{\partial\psi_\lambda}{\partial z}\to 0\ \text{and}~ \ \frac{1}{r}\frac{\partial\psi_\lambda}{\partial z}\  \text{approaches a finite limit}.
\end{equation*}
\item[(v)]Let $a(\lambda)$ be any point of ~~$\Omega_\lambda$. Then, as $\lambda \to +\infty$,
\begin{equation*}
  K\zeta_\lambda(\cdot)-{K(\cdot,a(\lambda))} \to 0 \ \  \text{in}\ \  W^{1,p}_{\text{loc}}(D),\ \ ~1\le p<2,
\end{equation*}
and hence in $L^r_{\text{loc}}(D)$, $1\le r<\infty$. Moreover, for any $\alpha\in (0,1)$, as $\lambda \to +\infty$,
\begin{equation*}
  K\zeta_\lambda(\cdot)-{K(\cdot,a(\lambda))} \to 0 \ \  \text{in}\ \  C^{1,\alpha}_{loc}(D\backslash\{(r_*,0)\}).
\end{equation*}
\end{itemize}
\end{theorem}

\begin{remark}
  Compared to \cite{FT}, our method seems more concise. We do not located the vortex by constraining the impulse of the flow to be a constant. The velocity at infinity of the flow is determined. One can also see that our result is consistent with the Kelvin-Hicks formula $\eqref{1-8}$. 
\end{remark}

\begin{remark}
  With these results in hand, one may expect to further study the asymptotic shape of the vortex core, see \cite{FT, Tur83} for instance.  For the regularity of $\partial \Omega_\lambda$, we address the reader to \cite{CF} for more discussion.
\end{remark}

Also, we study vortex rings outside a ball. The approach used here is a little different from the previous cases.

\begin{theorem}\label{thm3}
Let $U=\{(r,\theta,z)\in \mathbb{R}^3~|~r^2+z^2>d^2\}$ for some $d>0$ and let $D=U\cap\Pi$. Then for every $W>0$ and all sufficiently large $\lambda$, there exists a weak solution of $\eqref{1-7}$ $(\psi_\lambda,\zeta_\lambda)$ satisfying
\begin{itemize}
\item[(i)]For any $p>1$, $0<\alpha<1$, $\psi_\lambda\in W^{2,p}_{\text{loc}}(D)\cap C^{1,\alpha}_{\text{loc}}(\bar{D})$ and satisfies
\begin{equation*}
  \mathcal{L}\psi_\lambda=\zeta_\lambda\ \ \text{a.e.} \ \text{in} \ D.
\end{equation*}

\item[(ii)] $(\psi_\lambda,\zeta_\lambda)$ is of the form
\begin{equation*}
  \begin{split}
    & \psi_\lambda=K\zeta_\lambda-\frac{W\log\lambda}{2}r^2+\frac{W\log\lambda}{2}\frac{r^2d^3}{(r^2+z^2)^{{3}/{2}}}-\mu_\lambda,\ \ \zeta_\lambda=\lambda I_{\Omega_\lambda}, \\
     & \Omega_\lambda=\{x\in D~|~\psi_\lambda(x)>0 \},\ \ \ \lambda|\Omega_\lambda|=1,
\end{split}
\end{equation*}
for some $\mu_\lambda>0$ depending on $\lambda$.

 \item[(iii)] For any $\alpha\in(0,1)$, if $\lambda$ is large enough, there holds
 $$diam(\Omega_\lambda) \le C_0 \lambda^{-\frac{\alpha}{2}},$$
 where the positive number $C_0$ is independent of $\lambda$ and $\alpha$.
 Moreover,
 \begin{equation*}
\begin{split}
   \lim_{\lambda\to +\infty} \frac{\log diam(\Omega_\lambda)}{\log (\lambda^{-\frac{1}{2}})} & =1, \\
   \lim_{\lambda \to +\infty}dist_{\mathcal{C}_{r_*}}(\Omega_\lambda)&=0.
\end{split}
\end{equation*}
where $r_*\in[d,+\infty)$ satisfies $\Gamma_2(r_*)=\max_{t\in[d,+\infty)}\Gamma_2(t)$ and
\begin{equation*}
  \Gamma_2(t):=t-8\pi^2Wt^2+\frac{8\pi^2Wd^3}{t},\ t\in(0,+\infty).
\end{equation*}
Furthermore, if $W<1/(24\pi^2d)$, then there exists a constant $R_0>1$ independent of $\lambda$ such that
$$diam(\Omega_\lambda)\le R_0\lambda^{-\frac{1}{2}}$$
provided $\lambda$ is large enough; meanwhile, as $\lambda\to+\infty$,
$$\mu_\lambda   = \Big(\frac{r_*}{8\pi^2}-\frac{Wr_*^2}{2}+\frac{Wd^3}{2r_*}\Big)\log\lambda+O(1).$$

\item[(iv)]Let $$\mathbf{v}_\lambda=\frac{1}{r}\Big(-\frac{\partial\psi_\lambda}{\partial z}\mathbf{e}_r+\frac{\partial\psi_\lambda}{\partial r}\mathbf{e}_z\Big),$$ then
\begin{equation*}
\begin{split}
    & \mathbf{v_\lambda}\cdot\mathbf{n}=0 \  \ \text{on}\ \partial U,\\
    & \mathbf{v}_\lambda\to -W\log\lambda~\mathbf{e}_z\ \  \text{at}\ \infty, \  \  \text{as}\  \lambda \to +\infty.
\end{split}
\end{equation*}
where $\mathbf{n}$ is the unit outward normal of $\partial U$. Moreover, as $r\to 0$,
\begin{equation*}
     \frac{1}{r}\frac{\partial\psi_\lambda}{\partial z}\to 0\ \text{and}~ \ \frac{1}{r}\frac{\partial\psi_\lambda}{\partial r}\  \text{approaches a finite limit}.
\end{equation*}

\item[(v)]Let $a(\lambda)$ be any point of ~~$\Omega_\lambda$. Then, as $\lambda \to +\infty$,
\begin{equation*}
  K\zeta_\lambda(\cdot)-{K(\cdot,a(\lambda))} \to 0 \ \  \text{in}\ \  W^{1,p}_{\text{loc}}(D),\ \ ~1\le p<2,
\end{equation*}
and hence in $L^r_{\text{loc}}(D)$, $1\le r<\infty$. Moreover, for any $\alpha  \in(0,1)$
\begin{equation*}
  K\zeta_\lambda(\cdot)-{K(\cdot,a(\lambda))} \to 0 \ \  \text{in}\ \  C^{1,\alpha}_{loc}(D\backslash\{(r_*,0)\}).
\end{equation*}
\end{itemize}
\end{theorem}

 The main difference in this case is that one cannot obtain the compactness by using embedding of sets of symmetric functions. Our strategy is as follows.  We first construct the vortex rings by constraining their supports on some bounded region. And then we show that those vortices will concentrate in the interior of this region. As one will see, this is enough for our purpose.

Our last result is on the existence of vortex rings in bounded domains.

\begin{theorem}\label{thm4}
Let $U=\{(r,\theta,z)\in \mathbb{R}^3~|~r^2+z^2<b^2\}$ or $\{(r,\theta,z)\in \mathbb{R}^3~|~0\le r<b, -c<z<c\}$, for some $b,c\in \mathbb{R}_+$. Let $D=U\cap\Pi$, then for every $\lambda>|D|^{-1}$, there exists a weak solution of $\eqref{1-7}$ $(\psi_\lambda,\zeta_\lambda)$ satisfying
\begin{itemize}
\item[(i)] For any $p>1$, $0<\alpha<1$, $\psi_\lambda\in W^{2,p}_{\text{loc}}(D)\cap C^{1,\alpha}(\bar{D})$  and satisfies
\begin{equation*}
  \mathcal{L}\psi_\lambda=\zeta_\lambda\ \ \text{a.e.} \ \text{in} \ D.
\end{equation*}

\item[(ii)] $(\psi_\lambda,\zeta_\lambda)$ is of the form
\begin{equation*}
  \begin{split}
       & \psi_\lambda=K\zeta_\lambda-\mu_\lambda,\ \ \zeta_\lambda=\lambda I_{\Omega_\lambda},\\
       & \Omega_\lambda=\{x\in D~|~\psi_\lambda(x)>0 \},\ \lambda|\Omega_\lambda|=1,
  \end{split}
\end{equation*}
for some $\mu_\lambda>0$ depending on $\lambda$.

 \item[(iii)] For any $\alpha\in(0,1)$, there holds $$diam(\Omega_\lambda) \le 4b \lambda^{-\frac{\alpha}{2}}$$
 provided $\lambda$ is large enough.
 Moreover,
 \begin{equation*}
\begin{split}
   \lim_{\lambda\to +\infty} \frac{\log diam(\Omega_\lambda)}{\log (\lambda^{-\frac{1}{2}})} & =1, \\
   \lim_{\lambda \to +\infty}dist_{\mathcal{C}_{b}}(\Omega_\lambda)&=0.
\end{split}
\end{equation*}
\item[(iv)]Let $$\mathbf{v}_\lambda=\frac{1}{r}\Big(-\frac{\partial\psi_\lambda}{\partial z}\mathbf{e}_r+\frac{\partial\psi_\lambda}{\partial r}\mathbf{e}_z\Big),$$ then
\begin{equation*}
\begin{split}
    & \mathbf{v_\lambda}\cdot\mathbf{n}=0 \  \ \text{on}\ \partial U,\\
\end{split}
\end{equation*}
where $\mathbf{n}$ is the unit outward normal of $\partial U$. Moreover, as $r\to 0$,
\begin{equation*}
     \frac{1}{r}\frac{\partial\psi_\lambda}{\partial z}\to 0\ \text{and}~ \ \frac{1}{r}\frac{\partial\psi_\lambda}{\partial r}\  \text{approaches a finite limit}.
\end{equation*}
 \item[(v)]Let $a(\lambda)$ be any point of ~~$\Omega_\lambda$. Then, as $\lambda \to +\infty$,
\begin{equation*}
  K\zeta_\lambda(\cdot)-{K(\cdot,a(\lambda))} \to 0 \ \  \text{in}\ \  W^{1,p}_0(D),\ \ ~1\le p<2,
\end{equation*}
and hence in $L^r(D)$, $1\le r<\infty$. Moreover, for any $\alpha\in (0,1)$, as $\lambda \to +\infty$,
\begin{equation*}
  K\zeta_\lambda(\cdot)-{K(\cdot,a(\lambda))} \to 0 \ \  \text{in}\ \  C^{1,\alpha}_{loc}(D).
\end{equation*}
\end{itemize}
\end{theorem}

\begin{remark}
In \cite{D2}, Dekeyser constructed a family of desingularized solutions to the lake equations in a bounded domain. He was mainly concerned with the desingularization of steady vortex pairs. Note that our asymptotic estimate is sharper than Theorem A of \cite{D2}. Our last result can be regarded as an improvement in the vortex ring case.
\end{remark}

\begin{remark}
  This problem has also been considered in \cite{BF2,Ta}, but the vorticity function $f$ need to be H$\ddot{\text{o}}$lder-continuous therein. Moreover, their limiting objects are degenerate vortex rings with vanishing circulation. Our result provides a desingularization of singular vortex filaments with nonvanishing vorticity.
\end{remark}

\begin{remark}
  The operator $\mathcal{L}$ also occurs in the plasma problem, see \cite{BB2, Te}. Caffarelli and Friedman in \cite{CF2} obtained some asymptotic estimates for this problem. They constructed a family of plasmas which were shown to converge to the part of the boundary of the domain. Note that the nonlinearity $f$ therein is different from ours.
\end{remark}

\section{Vortex Rings in a Cylinder and the Whole Space}

To begin with, we need some estimates for the Green's function of $\mathcal{L}$ in $D$.
\begin{lemma}\label{le1}
  Let D be admissible, then we have
\begin{equation}\label{288}
  K(r,z,r',z')=G(r,z,r',z')-H(r,z,r',z'),
\end{equation}
where
\begin{equation}\label{007}
  G(r,z,r',z')=\frac{rr'}{8\pi^2}\int_{-\pi}^{\pi}\frac{cos\theta'd\theta'}{[(z-z')^2+r^2+r'^2-2rr'cos\theta']^\frac{1}{2}},
\end{equation}
and  $H(r,z,r',z')\in C^{\infty}(D\times D)$ is non-negative. Moreover, define
\begin{equation}\label{300}
  \sigma=[(r-r')^2+(z-z')^2]^\frac{1}{2}/(4rr')^\frac{1}{2},
\end{equation}
then for all $\sigma > 0$
\begin{equation}\label{Tadiewrong}
  0<K(r,z,r',z')\leq G(r,z,r',z')\leq\frac{(rr')^\frac{1}{2}}{8\pi^2}\sinh^{-1}(\frac{1}{\sigma}).
\end{equation}
\end{lemma}

\begin{proof}
  The calculation of Green's function is standard, one can refer to \cite{Ta}. We prove $\eqref{Tadiewrong}$ here. Indeed, by  Lemma 3.3 of \cite{Ta}, one has
\begin{equation*}
\begin{split}
   G(r,z,r',z') & \le \frac{(rr')^\frac{1}{2}}{32\pi^2}\int_{0}^{\pi}(\sigma^2+\sin^2\frac{\theta}{2})^{-3/2}\sin^2\theta d\theta \\
                & \le \frac{(rr')^\frac{1}{2}}{8\pi^2}\int_{0}^{1}(\sigma^2+t^2)^{-3/2}t^2(1-t^2)^{1/2}dt \\
                & \le \frac{(rr')^\frac{1}{2}}{8\pi^2}\int_{0}^{1}(\sigma^2+t^2)^{-3/2}t^2dt    \\
                & \le \frac{(rr')^\frac{1}{2}}{8\pi^2}\sinh^{-1}(\frac{1}{\sigma}).
\end{split}
\end{equation*}
The proof is completed.
\end{proof}

 Using the asymptotic behaviour of the Green's function (see \cite{Fr1}), we have the following estimate
\begin{lemma}\label{le2}
  Let $G$ be defined by $\eqref{007}$. For any $l>0$ and small positive number ~$\epsilon$(say $\epsilon<1/2$), there exist $C_1$ and $C_2$ such that
\begin{equation}\label{289}
   G(r,z,r',z')r'\le\frac{(1+C_1\epsilon)l^2}{4\pi^2}  \log[(r-r')^2+(z-z')^2]^{-\frac{1}{2}}+C_2,
\end{equation}
for all $(r,z),(r',z')\in D^{\epsilon}_{l}=\{(r,z)\in D| (r-l)^2+z^2\le(l\epsilon)^2 \},$
where the positive numbers $C_1$,$C_2$ depend only  on the upper bounded of $l$, but not on $\epsilon$.
\end{lemma}

\begin{proof}
  This is a direct consequence of the asymptotic behaviour of the Green's function in \cite{Fr1}. Introduce new coordinates about $(l,0)$:
\[\frac{r}{l}-1=\epsilon X=\epsilon S \cos T,\  \frac{z}{l}=\epsilon Y=\epsilon S \sin T,\]
Let $\textbf{S}$ denote the coordinate pair $(S,T)$ and~ $ |\textbf{S}-\textbf{S}'|:=\{S^2+S'^2-2SS'\cos(T-T')\}^\frac{1}{2}$.~Then the function $G(r,z,r',z')r'$ has an expansion
 \begin{equation}\label{299}
\begin{split}
    G(r,z,r',z')r'&=\frac{l^2}{4\pi^2}\log\frac{1}{l\epsilon|\textbf{S}-\textbf{S}'|}\{1+\sum_{n=1}^{\infty}\epsilon^np_n(\textbf{S},\textbf{S}')\\
                  &+\frac{l^2}{4\pi^2}\log(8l)\{1+\sum_{n=1}^{\infty}\epsilon^np_n(\textbf{S},\textbf{S}')\}
                  +\frac{l^2}{4\pi^2}\{-2+\sum_{n=1}^{\infty}\epsilon^nP_n(\textbf{S},\textbf{S}')\},
\end{split}
\end{equation}
where the $p_n$ and $P_n$ are homogeneous polynomials of degree $n$ in $X$ and $X'$. The two series in $\eqref{299}$ converge uniformly and absolutely for
\[\epsilon(S+S')\le 2-\alpha,\ \epsilon |X|\le 1-\alpha,\ \epsilon |X'|\le 1-\alpha \ (\alpha>0).  \]
From this $\eqref{289}$ clearly follows.
\end{proof}

\begin{remark}
   One can also get $\eqref{289}$ directly from the expression for the Green's function, but we prefer to use the asymptotic expansion above, because $\eqref{299}$ is interesting in itself.
\end{remark}

The following result is a variant of Lemma 6 of Burton \cite{B5}.
\begin{lemma}\label{burton}
  Let $D\subseteq \Pi$ be a domain, let $(\psi,\zeta)\in W^{2,p}_{\text{loc}}(D)\times L^\infty(D)$ for some $p>1$ satisfying $\mathcal{L}\psi=\zeta$ a.e. in $D$. Suppose that $\zeta=f\circ \psi$ a.e. in $D$, for some monotonic function $f$. Then $\text{div}~(\zeta \nabla^\bot\psi)=0$ as a distribution.
\end{lemma}

\begin{proof}
  The proof is the same as in \cite{B5}, however we repeat it here for the sake of completeness.

  We can choose a sequence $\{f_n\}$ of bounded Lipschitz functions on $\mathbb{R}$, such that $|f_n|\le |f|$ everywhere, and $f_n(t)\to f(t)$ at every point $t$ where $f: \mathbb{R}\to [-\infty,+\infty]$ is continuous; moreover, $f_n(t)\to f(t)$ at points $t$ where $f(t)=0$. Now we have in the  sense of distribution
  \begin{equation*}
    div((f_n\circ\psi)\nabla^\bot\psi)=(f_n\circ\psi)div(\nabla^\bot\psi)+(f'_n\circ\psi)\nabla\psi\cdot\nabla^\bot\psi=0.
  \end{equation*}
  Given $\varphi\in C_0^\infty(D)$, we now have
  \begin{equation*}
    \int_D(f_n\circ\psi)\nabla^\bot\psi\cdot\nabla\varphi=0.
  \end{equation*}
  Since $|(f_n\circ\psi)\nabla^\bot\psi|\le|\zeta\nabla^\bot\psi|$ almost everywhere, $\zeta\nabla^\bot\psi \in L_{\text{loc}}^1$, and $(f_n\circ\psi)\nabla^\bot\psi \to \zeta\nabla^\bot\psi$ almost everywhere on $D \backslash\psi^{-1}(A)$, where $A\subseteq\mathbb{R}$ is the set of discontinuities $t$ of $f$ with $f(t)\neq 0$. Moreover $A$ is countable, and for each $t\in A$ we have $\zeta=\mathcal{L}\psi=0$ a.e. on $\psi^{-1}(t)$, so $\textit{m}_2(\psi^{-1}(t))=0$. Hence $\textit{m}_2(\psi^{-1}(A))=0$. By the dominated convergence theorem, we conclude that
  \begin{equation*}
    \int_D(f\circ\psi)\nabla^\bot\psi\cdot\nabla\varphi=0,\ \ \forall\/ \varphi\in C_0^\infty(D),
  \end{equation*}
  which completes our proof.
\end{proof}

Let $D\subseteq\mathbb{R}^2$ be a domain. For a measurable function $\zeta_0 \ge 0$ in $D$ such that $\nu(\zeta_0^{-1}[t,+\infty])<+\infty$ for all $t>0$, the essentially unique non-negative decreasing rearrangement $\zeta_0^\bigtriangleup$ of $\zeta_0 $ is defined on $[0,+\infty)$ such that
$${\textit{m}}_1\{s\in[0,+\infty)|~\zeta_0^\bigtriangleup(s)\ge t \}= \nu\{x \in D|~\zeta_0(x)\ge t \}\ \    \text{for all}\ \   t>0,$$
Let $\zeta_0=\lambda I_A$  be measurable on $D$ such that  $\int_D \zeta d\nu=1$. Define sets $\mathcal{R}_\lambda\subseteq
\mathcal{RC}_\lambda\subseteq\mathcal{WR}_\lambda$ as follows,
\begin{equation*}
\begin{split}
\mathcal{R}_\lambda &= \{\zeta\in L^\infty(D)~|~\zeta^\bigtriangleup=\zeta_0^\bigtriangleup            \}, \\
\mathcal{RC}_\lambda &=\{\zeta\in L^\infty(D)~|~ \zeta= \lambda I_A \ \text{for some measurable subset}\  A\subseteq D , \int_D \zeta d\nu \le 1\},              \\
\mathcal{WR}_\lambda &= \{\zeta\in L^\infty(D)~|~ \zeta \ge 0, \int_{D}(\zeta-t)^+d\nu \le \int_{D}(\zeta_0-t)^+d\nu,\ \forall~ t>0  \},
\end{split}
\end{equation*}
where superscript $+$ signifies the positive part.

Define the impulse of the flow as follows
$$\mathcal{I}(\zeta)=\frac{1}{2}\int_{D}r^2\zeta d\nu.$$

\noindent \textbf{3.1.~Vortex Rings in a Cylinder}

We start with the case of infinite cylinder. So we take $D=\{(r,z)\in \mathbb{R}^2~|~0<r<d, z\in \mathbb{R}\}$ for some fixed $d>0$.

For fixed $W>0$ and $\lambda>1$, inspired by the Kelvin-Hicks formula \eqref{1-8}, we consider the energy as follows
\[E_\lambda(\zeta)=\frac{1}{2}\int_D{\zeta K\zeta}d\nu-\frac{W\log\lambda}{2}\int_{D}r^2\zeta d\nu.\]

\begin{lemma}\label{le8}
For any fixed $W>0$ and $\lambda>1$, there exists $\zeta=\zeta_\lambda \in \mathcal{RC}_\lambda $ such that
\begin{equation*}
 E_\lambda(\zeta)= \max_{\tilde{\zeta} \in \mathcal{WR}_\lambda}E_\lambda(\tilde{\zeta}).
\end{equation*}
For any maximizer $\zeta_\lambda$,~we have $K\zeta_\lambda\in W^{2,p}_{\text{loc}}(D)\cap C^{1,\alpha}(\bar{D})$ for any $p>1$, $0<\alpha<1$. Moreover,
\begin{equation*}
\zeta_\lambda=\lambda I_{\Omega_\lambda} \ \ a.e.\  \text{in}\  D,
\end{equation*}
where
\begin{equation*}
  \Omega_\lambda=\{(r,z)\in D~|~K\zeta_\lambda(r,z)-\frac{W\log\lambda}{2}r^2>\mu_\lambda \},
\end{equation*}
and the Lagrange multiplier $\mu_\lambda\ge 0$ is determined by $\zeta_\lambda$. If $\zeta_\lambda\neq0$ and $\mu_\lambda>0$, then $\zeta_\lambda \in \mathcal{R}_\lambda$.
\end{lemma}

\begin{proof}
We refer to \cite{Dou2} for the proofs.
\end{proof}

\begin{remark}
  $\mu_\lambda$ is also called the fluid constant which is a constituent of the flow rate between the boundaries $\partial D$ and $\partial{\Omega_\lambda}$.
\end{remark}

Using the asymptotic behaviour of the Green's function, we can obtain the lower bound estimate of energy.
\begin{lemma}\label{le9}
  For any $a\in(0,d)$, there exists $C>0$ such that for all $\lambda$ sufficiently large,~we have
\begin{equation*}
  E_\lambda(\zeta_\lambda)\ge (\frac{a}{16\pi^2}-\frac{Wa^2}{2})\log\lambda-C,
\end{equation*}
where the positive number $C$ depends only on $a$, but not on $\lambda$.
\end{lemma}
\begin{proof}
 Choose a test function $\hat{\zeta}_\lambda \in \mathcal{R}_\lambda$ defined by $\hat{\zeta}_\lambda=\lambda I_{B_\varepsilon((a,0))}$ with $2\pi^2a\varepsilon^2\lambda=1$.~Since $\zeta_\lambda$ is a maximizer,~we have $E(\zeta_\lambda)\ge E(\hat{\zeta}_\lambda)$. By $\eqref{288}$ and $\eqref{299}$, we obtain
 \begin{equation*}
\begin{split}
    E(\hat{\zeta}_\lambda)&\ge \frac{1}{2}\int_{D}\int_{D}\hat{\zeta}(r,z)G(r,z,r',z')\hat{\zeta}(r',z')4\pi^2r'r dr'dz'drdz-\frac{Wa^2}{2}\log\lambda-C_1\\
                  &=\frac{a^2}{2}\int_{|S|\le a^{-1}}\int_{|S'|\le a^{-1}}\lambda^2\{(1+O(\varepsilon))\log\frac{1}{\varepsilon}+O(1)\} (a+O(\varepsilon)) (a\varepsilon)^4dXdYdX'dY'\\
                  &\ \ \ -\frac{Wa^2}{2}\log\lambda-C_1\\
                  &\ge(\frac{a}{16\pi^2}-\frac{Wa^2}{2})\log\lambda-C.
\end{split}
\end{equation*}
The proof is completed.
\end{proof}

Next we estimate the energy of the vortex core. Set $\psi_\lambda=K\zeta_\lambda-W(\log \lambda) r^2/2-\mu_\lambda$.~The kinetic energy of the vortex core is defined as follows
\begin{equation*}
  J(\zeta_\lambda)=\frac{1}{2}\int_D\frac{{|\nabla (\psi_\lambda^+)|^2}}{r^2}d\nu\
\end{equation*}
where ${\psi_\lambda^+} = \max \{\psi_\lambda,0\}$.

\begin{lemma}\label{le10}
  $J(\zeta_\lambda)\le C$.
\end{lemma}

\begin{proof}
Firstly note that $K\zeta_\lambda$ satisfies the following elliptic equation
\begin{equation*}
  \mathcal{L}(K\zeta_\lambda)=\zeta_\lambda.
\end{equation*}
Since $\psi_\lambda^+\in H(D)$, we can take $\psi_\lambda^+$ as a test function. Therefore,
$$\int_{D}\frac{\nabla(K\zeta_\lambda) \cdot \nabla{\psi_\lambda^+}}{r^2}d\nu=\int_{D}\psi_\lambda^+ \zeta_\lambda d\nu.$$
Thus
\begin{equation}\label{3077}
\begin{split}
   2J(\zeta_\lambda)&= \int_D\frac{{|\nabla (\psi_\lambda^+)|^2}}{r^2}d\nu = \int_{D}\psi_\lambda^+ \zeta_\lambda d\nu          \\
     &= \lambda \int_{\Omega_\lambda}\psi_\lambda^+d\nu \\
     &\le \lambda |\Omega_\lambda|^{\frac{1}{2}}(\int_{\Omega_\lambda}(\psi_\lambda^+)^2d\nu)^{\frac{1}{2}}\\
     &\le C\lambda |\Omega_\lambda|^{\frac{1}{2}}(\int_{D}(\psi_\lambda^+)^2drdz)^{\frac{1}{2}}\\
     &\le C\lambda |\Omega_\lambda|^{\frac{1}{2}}\int_{D}|\nabla{\psi_\lambda^+}|drdz\\
     &\le C\lambda |\Omega_\lambda|(\int_D\frac{{|\nabla (\psi_\lambda^+)|^2}}{r^2}d\nu)^{\frac{1}{2}} \\
     &\le C(J(\zeta_\lambda))^{\frac{1}{2}},
\end{split}
\end{equation}
From $\eqref{3077}$ we conclude the desired result.
\end{proof}

We are now ready to estimate the Lagrange multiplier $\mu_\lambda$. Notice that
\begin{equation*}
\begin{split}
    2E_\lambda(\zeta_\lambda) & =\int_{D} \zeta_\lambda (K\zeta_\lambda-W(\log \lambda) r^2) d\nu \\
     & =\int_{D} \zeta_\lambda (K\zeta_\lambda-W(\log \lambda) r^2/2-\mu_\lambda) d\nu-W(\log \lambda)\mathcal{I(\zeta_\lambda)}+\mu_\lambda\\
     &=2J(\zeta_\lambda)-W(\log \lambda)\mathcal{I(\zeta_\lambda)}+\mu_\lambda.
\end{split}
\end{equation*}
Thus
\begin{lemma}\label{le11}
  $\mu_\lambda=2E_\lambda(\zeta_\lambda)+W(\log \lambda)\mathcal{I(\zeta_\lambda)}+O(1)$, as $\lambda \to +\infty$.
\end{lemma}

\begin{corollary}\label{le12}
If $\lambda$ is large enough, then $\zeta_\lambda \in \mathcal{R}_\lambda$.
\end{corollary}

Our focus is the asymptotic behaviour of $\zeta_\lambda$ when $\lambda\to +\infty$. Hence we may and shall assume $\lambda$ is sufficiently large. Note that, by the Lemma 2 of \cite{B1}, if $\zeta_\lambda$ is a maximizer then so is $(\zeta_\lambda)^*$. Henceforth we shall assume $\zeta_\lambda=(\zeta_\lambda)^*$ and $K\zeta_\lambda$ is a symmetrically deceasing function on each line parallel to the $z$ axis. Throughout the sequel we shall denote $C,C_1,C_2, ..., $ for positive constants independent of $\lambda$.

We introduce the function $\Gamma_1$ as follows
\begin{equation*}
  \Gamma_1(t)=t-8\pi^2Wt^2,\ t\in[0,d].
\end{equation*}
Denote
\[
r_*=\left\{
   \begin{array}{lll}
        \frac{1}{16\pi^2 W} &    \text{if} & W>1/(16\pi^2d), \\
         d                  &    \text{if} & W\le1/(16\pi^2d).
    \end{array}
   \right.
\]
It is easy to check that $\Gamma_1(r_*)=\max_{t\in[0,d]}\Gamma_1(t)$.

We now study the asymptotic behaviour of the vortex core.

Denote $A_\lambda=\inf\{r|(r,0)\in \Omega_\lambda\}$,\ $B_\lambda=\sup\{r|(r,0)\in \Omega_\lambda\}$.
\begin{lemma}\label{le13}
  $\lim_{\lambda \to +\infty}A_\lambda=r_*$.
\end{lemma}

\begin{proof}
Let $\sigma$ be defined by $\eqref{300}$, let $\alpha \in (0,1)$ and $\epsilon:=\lambda^{-\frac{1}{2}}$. For any point $(r_\lambda,z_\lambda)\in \bar{\Omega}_\lambda$, we have
\begin{equation*}
\begin{split}
    K\zeta_\lambda(r_\lambda,z_\lambda)&\le \int_{D}G(r_\lambda,z_\lambda,r',z')\zeta_\lambda(r',z')2\pi r'dr'dz'\\
                       &=(\int_{D\cap\{\sigma>\epsilon^\alpha\}}+\int_{D\cap\{\sigma \le \epsilon^\alpha\}})G(r_\lambda,z_\lambda,r',z')\zeta_\lambda(r',z')2\pi r'dr'dz'\\
                       &:=I_1+I_2.
\end{split}
\end{equation*}
For the first term $I_1$, we can use $\eqref{Tadiewrong}$ to obtain
\begin{equation}\label{900}
\begin{split}
   I_1 &=\int_{D\cap\{\sigma>\epsilon^\alpha\}}G(r_\lambda,z_\lambda,r',z')\zeta_\lambda(r',z')2\pi r'dr'dz'\\
       &\le \frac{(r_\lambda)^{\frac{1}{2}}}{4\pi} \sinh^{-1}(\frac{1}{\epsilon^\alpha})\int_{D\cap\{\sigma>\epsilon^\alpha\}}\zeta_\lambda(r',z')r'^{\frac{3}{2}}dr'dz'\\
       &\le \frac{d}{8\pi^2} \sinh^{-1}(\frac{1}{\epsilon^\alpha}).
\end{split}
\end{equation}
On the other hand, notice that $D\cap\{\sigma \le \epsilon^\alpha\} \subseteq B_{2d\epsilon^\alpha}((r_\lambda,z_\lambda))$, hence Lemma $\ref{le2}$ yields that

\begin{equation*}
\begin{split}
   I_2 &=\int_{D\cap\{\sigma \le \epsilon^\alpha\}}G(r_\lambda,z_\lambda,r',z')\zeta_\lambda(r',z')2\pi r'dr'dz'\\
       &\le \frac{(1+6C\epsilon^\alpha)(r_\lambda)^2}{2\pi}\int_{B_{2d\epsilon^\alpha}((r_\lambda,z_\lambda))}\log [(r_\lambda-r')^2+(z_\lambda-z')^2]^{-\frac{1}{2}}\zeta_\lambda(r',z')dr'dz'+C\\
       &\le \frac{(1+6C\epsilon^\alpha)(r_\lambda)^2\log \lambda}{4\pi}\int_{B_{2d\epsilon^\alpha}((r_\lambda,z_\lambda))}\zeta_\lambda dr'dz'+C\lambda^{-\frac{\alpha}{2}}+C\\
       & \le \frac{r_\lambda\log\lambda}{8\pi^2}\int_{B_{2d\epsilon^\alpha}((r_\lambda,z_\lambda))}\zeta_\lambda d\nu+C\lambda^{-\frac{\alpha}{2}}\log \lambda+C.
\end{split}
\end{equation*}
Hence
\begin{equation}\label{317}
  K\zeta_\lambda(r_\lambda,z_\lambda)\le \frac{r_\lambda\log\lambda}{8\pi^2}\int_{B_{6\epsilon^\alpha}((r_\lambda,z_\lambda))}\zeta_\lambda d\nu+C\lambda^{-\frac{\alpha}{2}}\log \lambda+C+\frac{1}{8\pi^2} \sinh^{-1}(\lambda^{\frac{\alpha}{2}}).
\end{equation}
Recall that $ K\zeta_\lambda(r_\lambda,z_\lambda)-\frac{W(r_\lambda)^2}{2}\log \lambda\ge \mu_\lambda$. By Lemmas $\ref{le9}$ and $\ref{le11}$, for any $a\in(0,d)$, we have
\begin{equation*}
    K\zeta_\lambda(r_\lambda,z_\lambda)-\frac{W(r_\lambda)^2}{2}\log \lambda \ge (\frac{a}{8\pi^2}-Wa^2)\log\lambda+W(\log \lambda)\mathcal{I(\zeta_\lambda)}-C.
\end{equation*}
Thus
\begin{equation*}
\begin{split}
   (\frac{a}{8\pi^2}-Wa^2)\log\lambda+W(\log \lambda)\mathcal{I(\zeta_\lambda)}\le&  ~\frac{r_\lambda\log\lambda}{8\pi^2}\int_{B_{2d\epsilon^\alpha}((r_\lambda,z_\lambda))}\zeta_\lambda d\nu+C\lambda^{-\frac{\alpha}{2}}\log \lambda\\
     &+\frac{d}{8\pi^2} \sinh^{-1}(\lambda^{\frac{\alpha}{2}}) -\frac{W(r_\lambda)^2}{2}\log \lambda+C,
\end{split}
\end{equation*}
Dividing both sides of the above inequality by $\log \lambda/(8\pi^2)$, we obtain
\begin{equation}\label{888}
\begin{split}
   \Gamma_1(a)&+4\pi^2W[2\mathcal{I(\zeta_\lambda)}-(r_\lambda)^2]\\
   &\le  r_\lambda \int_{B_{2d\epsilon^\alpha}((r_\lambda,z_\lambda))}\zeta_\lambda d\nu-8\pi^2W(r_\lambda)^2+\frac{d\sinh^{-1}(\lambda^{\frac{\alpha}{2}})}{\log\lambda}+C\lambda^{-\frac{\alpha}{2}}+\frac{C}{\log\lambda}   \\
     &\le \Gamma_1(r_\lambda)+\frac{d\sinh^{-1}(\lambda^{\frac{\alpha}{2}})}{\log\lambda}+C\lambda^{-\frac{\alpha}{2}}+\frac{C}{\log\lambda}.
\end{split}
\end{equation}
Notice that
\begin{equation}\label{889}
  2\mathcal{I(\zeta_\lambda)}=\int_{D}r^2\zeta_\lambda d\nu \ge (A_\lambda)^2.
\end{equation}
Taking $r_\lambda=A_\lambda$ and $z_\lambda=0$, $\eqref{888}$ leads to
\begin{equation*}
  \Gamma_1(a)\le \Gamma_1(A_\lambda)+\frac{d\sinh^{-1}(\lambda^{\frac{\alpha}{2}})}{\log\lambda}+C\lambda^{-\frac{\alpha}{2}}+\frac{C}{\log\lambda}.
\end{equation*}
Now letting $\lambda$ tend to $+\infty$, we deduce that
\begin{equation*}
  \Gamma_1(a)\le \liminf_{\lambda\to +\infty}\Gamma_1(A_\lambda)+\frac{d\alpha}{2}.
\end{equation*}
Hence we get the desired result by letting $a\to r_*$ and $\alpha \to 0$.
\end{proof}

The next lemma gives the limit of the impulse of the flow.
\begin{lemma}\label{le14}
$\lim_{\lambda \to +\infty}\mathcal{I(\zeta_\lambda)}=r_*^2/2.$
\end{lemma}

\begin{proof}
From $\eqref{888}$ and $\eqref{889}$, we know that for any $\alpha \in(0,1)$
\begin{equation*}
  0\le \liminf_{\lambda\to +\infty}(2\mathcal{I(\zeta_\lambda)}-(A_\lambda)^2)\le \limsup_{\lambda\to +\infty}(2\mathcal{I(\zeta_\lambda)}-(A_\lambda)^2)\le \alpha/2.
\end{equation*}
Hence
$$\lim_{\lambda \to +\infty}\mathcal{I(\zeta_\lambda)}=\lim_{\lambda \to +\infty}(A_\lambda)^2/2=r_*^2/2.$$
\end{proof}

From the limit of the impulse of the flow, we can deduce that most of the vortex core will concentrate.
\begin{lemma}\label{le15}
 For any $ \eta >0$, there holds
$$\lim_{\lambda \to +\infty} \int_{D\cap\{{r\ge r_*+\eta}\}}\zeta_\lambda d\nu=0.$$
\end{lemma}

\begin{proof}
We give the proof by contradiction. Suppose that there exist $\eta_0>0, \kappa_0>0$ satisfying
$$\lim_{\lambda \to +\infty} \int_{D\cap\{{r\ge r_*+\eta_0}\}}\zeta_\lambda d\nu=\kappa_0>0.$$
Note that
\begin{equation*}
\begin{split}
   2\mathcal{I(\zeta_\lambda)}&=\int_{D\cap\{{r< r_*+\eta_0}\}}r^2\zeta_\lambda d\nu+\int_{D\cap\{{r\ge r_*+\eta_0}\}}r^2\zeta_\lambda d\nu \\
                              &\ge(A_\lambda)^2\int_{D\cap\{{r< r_*+\eta_0}\}}\zeta_\lambda d\nu+(r_*+\eta_0)^2\int_{D\cap\{{r\ge r_*+\eta_0}\}}\zeta_\lambda d\nu.
\end{split}
\end{equation*}
Letting $\lambda\to +\infty$, we deduce that
$$r_*^2=\lim_{\lambda \to +\infty}2\mathcal{I(\zeta_\lambda)}\ge r_*^2(1-\kappa_0)+(r_*+\eta_0)^2\kappa_0>r_*^2,$$
which leads to a contradiction and thus completes our proof.
\end{proof}

By Lemma $\ref{le15}$, we can further deduce the following

\begin{lemma}\label{le16}
  $\lim_{\lambda \to +\infty}B_\lambda=r_*$.
\end{lemma}

\begin{proof}
We may assume $0<r_*<d$, otherwise we have done by Lemma $\ref{le13}$. From $\eqref{888}$, we conclude that for any $\alpha>0$, $\epsilon:=\lambda^{-\frac{1}{2}}$,
\begin{equation*}
\Gamma_1(r_*)+4\pi^2Wr_*^2+4\pi^2W\liminf_{\lambda \to +\infty}B_\lambda\le \liminf_{\lambda \to +\infty}\int_{B_{2d\epsilon^\alpha}((B_\lambda,0))}\zeta_\lambda d\nu+\frac{d\alpha}{2}.
\end{equation*}
Since $\liminf_{\lambda \to +\infty}(B_\lambda)^2\ge r_*^2$, we get
\begin{equation*}
  r_*\le \liminf_{\lambda \to +\infty}\int_{B_{2d\epsilon^\alpha}((B_\lambda,0))}\zeta_\lambda d\nu+\frac{d\alpha}{2}.
\end{equation*}
Taking $\alpha$ so small such that $r_*-d\alpha/2>0$, the desired result follows from Lemma $\ref{le15}$.~
\end{proof}

The next lemma give an estimate of the cross-section $\Omega_\lambda$, which implies that the vortex core will shrink to the circle $\mathcal{C}_{r_*}$.
\begin{lemma}\label{le17}
For any positive number $\alpha < 1$ , there holds $diam(\Omega_\lambda) \le 4d \lambda^{-\frac{\alpha}{2}}$ provided $\lambda$ is large enough. Consequently, $\lim_{\lambda \to +\infty}dist_{\mathcal{C}_{r_*}}(\Omega_\lambda)=0$.
\end{lemma}

\begin{proof}
Let us use the same notations as in the proof of Lemma $\ref{le13}$. Recalling that $\int_{D}\zeta_\lambda d\nu=1$, so it suffices to prove that if $(r_\lambda,z_\lambda)\in \bar{\Omega}_\lambda$, then $\int_{B_{2d\epsilon^\alpha}((r_\lambda,z_\lambda))}\zeta_\lambda d\nu>1/2$.
From Lemma $\ref{le13}$ and Lemma $\ref{le16}$, we know
$$\lim_{\lambda \to +\infty}r_\lambda=r_*.$$
By $\eqref{888}$,
\begin{equation*}
 r_* \le r_*\liminf_{\lambda \to +\infty}\int_{B_{2d\epsilon^\alpha}((r_\lambda,z_\lambda))}\zeta_\lambda d\nu+d\alpha/2,
\end{equation*}
that is
\begin{equation}\label{318}
\liminf_{\lambda \to +\infty}\int_{B_{2d\epsilon^\alpha}((r_\lambda,z_\lambda))}\zeta_\lambda d\nu \ge 1-d\alpha/(2r_*).
\end{equation}
Therefore we have proved the result for all small $\alpha$ satisfying $1-d\alpha/(2r_*)>1/2$.
From this we may deduce that $diam(\Omega_\lambda) \le C/\log\lambda$.
Now $\eqref{900}$ can be refined as follows,
\begin{equation}\label{319}
\begin{split}
   I_1 &=\int_{D\cap\{\sigma>\epsilon^\alpha\}}G(r_\lambda,z_\lambda,r',z')\zeta_\lambda(r',z')2\pi r'dr'dz'\\
       &\le \frac{(r_\lambda)^{\frac{1}{2}}}{4\pi} \sinh^{-1}(\frac{1}{\epsilon^\alpha})\int_{D\cap\{\sigma>\epsilon^\alpha\}}\zeta_\lambda(r',z')r'^{\frac{3}{2}}dr'dz'\\
       &\le \frac{r_\lambda}{8\pi^2} \sinh^{-1}(\frac{1}{\epsilon^\alpha})\int_{D\cap\{\sigma>\epsilon^\alpha\}}\zeta_\lambda(r',z')d\nu+C\frac{\sinh^{-1}({1}/{\epsilon^\alpha})}{\log\lambda}.
\end{split}
\end{equation}
With $\eqref{319}$ in hand, we can repeat the proof and then improve $\eqref{318}$ as follows
\begin{equation}
\liminf_{\lambda \to +\infty}\int_{B_{2d\epsilon^\alpha}((r_\lambda,z_\lambda))}\zeta_\lambda d\nu \ge 1-\alpha/2,
\end{equation}
which finishes the proof.
\end{proof}

Summarizing:
\begin{proposition}\label{le18}
 For any $\alpha\in(0,1)$, there holds
 \begin{equation*}
   diam(\Omega_\lambda) \le 4d \lambda^{-\frac{\alpha}{2}}
 \end{equation*}
 provided $\lambda$ is large enough.
Moreover, one has
\begin{equation*}
\begin{split}
   \lim_{\lambda\to +\infty} \frac{\log diam(\Omega_\lambda)}{\log (\lambda^{-\frac{1}{2}})} & =1, \\
   \lim_{\lambda \to +\infty}dist_{\mathcal{C}_{r_*}}(\Omega_\lambda)&=0.
\end{split}
\end{equation*}
\end{proposition}

If $0<r_*<d$, then we can also get the following asymptotic estimates.
\begin{lemma}\label{asym1}
Suppose $ W> 1/(16\pi^2d)$, then as $\lambda\to +\infty$,
\begin{equation}\label{as1}
        E_\lambda(\zeta_\lambda)   = (\frac{r_*}{16\pi^2}-\frac{Wr_*^2}{2})\log\lambda+O(1),
\end{equation}
\begin{equation}\label{as2}
   \mu_\lambda             = (\frac{r_*}{8\pi^2}-\frac{Wr_*^2}{2})\log\lambda+O(1).
\end{equation}
\end{lemma}

\begin{proof}
  Recall that Lemma $\ref{le11}$, it suffices to prove $\eqref{as1}$. Indeed, by Proposition $\ref{le18}$, we have
\[{supp}(\zeta_\lambda)\subseteq B_{4d\lambda^{-\frac{1}{4}}}((A_\lambda,0))\]
for all sufficiently large $\lambda$. Hence Lemma $\ref{le2}$ yields that
\begin{equation*}
\begin{split}
  \int_D\zeta_\lambda K \zeta_\lambda d\nu & \le\frac{(A_\lambda)^2}{4\pi^2}\int_D\log[(r-r')^2+(z-z')^2]^{-\frac{1}{2}}\zeta_\lambda(r,z)\zeta_\lambda(r',z')4\pi^2rdrdzdr'dz'+O(1) \\
     & \le (A_\lambda)^2B_\lambda \int_D\log[(r-r')^2+(z-z')^2]^{-\frac{1}{2}}\zeta_\lambda(r,z)\zeta_\lambda(r',z')drdzdr'dz'+O(1).
\end{split}
\end{equation*}
By Lemma 4.2 of \cite{Tur83}, we have
\begin{equation*}
  \int_D\log[(r-r')^2+(z-z')^2]^{-\frac{1}{2}}\zeta_\lambda(r,z)\zeta_\lambda(r',z')drdzdr'dz'\le\frac{\log\lambda}{2}(\int_D\zeta_\lambda drdz)^2+O(1).
\end{equation*}
Note that
\begin{equation*}
  \int_D\zeta_\lambda drdz\le \frac{1}{2\pi A_\lambda}.
\end{equation*}
Thus
\begin{equation*}
   \int_D\zeta_\lambda K \zeta_\lambda d\nu\le \frac{B_\lambda}{8\pi^2}\log\lambda+O(1).
\end{equation*}
We then conclude that
\begin{equation*}
  E_\lambda(\zeta_\lambda)\le (\frac{B_\lambda}{16\pi^2}-\frac{W(B_\lambda)^2}{2})\log\lambda+O(1) \le (\frac{r_*}{16\pi^2}-\frac{Wr_*^2}{2})\log\lambda+O(1).
\end{equation*}
Combining this and Lemma $\ref{le9}$, we get $\eqref{as1}$. The proof is completed.
\end{proof}

Using Lemma $\ref{asym1}$, we have the following improvement of the first estimate in Proposition $\ref{le18}$ .
\begin{lemma}\label{sharp}
  Suppose $ W> 1/(16\pi^2d)$, then there exists a constant $R_0>1$ independent of $\lambda$ such that $diam(\Omega_\lambda)\le R_0\lambda^{-\frac{1}{2}}$.
\end{lemma}

In the proof of Lemma $\ref{sharp}$, we follow the strategy of \cite{Tur83}.
\begin{proof}
Let us use the same notations as in the proof of Lemma $\ref{le13}$. By Proposition $\ref{le18}$, we have
\[{supp}(\zeta_\lambda)\subseteq B_{4d\epsilon^{\frac{1}{4}}}((A_\lambda,0)).\]
Let $R>1$ to be determined. By Lemma $\ref{le2}$, we have
\begin{equation}\label{600}
\begin{split}
   I_1 &=\int_{D\cap\{\sigma>R\epsilon\}}G(r_\lambda,z_\lambda,r',z')\zeta_\lambda(r',z')2\pi r'dr'dz'\\
       &\le \frac{(1+C_1\epsilon^\frac{1}{4})(A_\lambda)^2}{2\pi}\int_{D\cap\{\sigma>R\epsilon\}}\log\frac{1}{[(r_\lambda-r')^2+(z_\lambda-z')^2]^{\frac{1}{2}}}\zeta_\lambda(r',z') dr'dz'\\
       &\le\frac{(1+C_1\epsilon^\frac{1}{4})(A_\lambda)^2}{2\pi}\int_{D\cap\{\sigma>R\epsilon\}}\log\frac{(4r_\lambda r')^\frac{1}{2}}{[(r_\lambda-r')^2+(z_\lambda-z')^2]^{\frac{1}{2}}}\zeta_\lambda(r',z')dr'dz'+C_2\\
       &\le \frac{A_\lambda}{8\pi^2}\log\frac{\lambda}{R}\int_{D\cap\{\sigma>R\epsilon\}}\zeta_\lambda(r',z')d\nu+C_3.
\end{split}
\end{equation}
and
\begin{equation}\label{601}
\begin{split}
   I_2 &=\int_{D\cap\{\sigma \le R\epsilon\}}G(r_\lambda,z_\lambda,r',z')\zeta_\lambda(r',z')2\pi r'dr'dz'\\
       &\le \frac{(1+C_4\epsilon^\frac{1}{4})(A_\lambda)^2}{2\pi}\int_{D\cap\{\sigma \le R\epsilon\}}\log\frac{1}{[(r_\lambda-r')^2+(z_\lambda-z')^2]^{\frac{1}{2}}}\zeta_\lambda(r',z')dr'dz'+C_5\\
       &\le \frac{(1+C_4\epsilon^\frac{1}{4})(A_\lambda)^2\log \lambda}{4\pi}\int_{D\cap\{\sigma \le R\epsilon\}}\zeta_\lambda dr'dz'+C_6\\
       & \le \frac{A_\lambda\log\lambda}{8\pi^2}\int_{D\cap\{\sigma \le R\epsilon\}}\zeta_\lambda d\nu+C_7.
\end{split}
\end{equation}
On the other hand, by $\eqref{as2}$, we have
\begin{equation}\label{602}
  \mu_\lambda\ge(\frac{A_\lambda}{8\pi^2}-\frac{W(A_\lambda)^2}{2})\log\lambda-C_8.
\end{equation}
Combine $\eqref{600}$, $\eqref{601}$ and $\eqref{602}$, we obtain
\begin{equation*}
  \frac{A_\lambda}{8\pi^2}\log\lambda \le \frac{A_\lambda}{8\pi^2}\log\frac{\lambda}{R}\int_{D\cap\{\sigma>R\epsilon\}}\zeta_\lambda(r',z')d\nu+\frac{A_\lambda\log\lambda}{8\pi^2}\int_{D\cap\{\sigma \le R\epsilon\}}\zeta_\lambda d\nu+C_9.
\end{equation*}
Hence
\begin{equation*}
  \int_{D\cap\{\sigma \le R\epsilon\}}\zeta_\lambda d\nu\ge 1-\frac{C_9}{\log R}.
\end{equation*}
Taking $R$ large enough such that $C_9(\log R)^{-1}<1/2$, we obtain
\begin{equation*}
  \int_{B_{2dR\epsilon}(r_\lambda,z_\lambda)}\zeta_\lambda d\nu\ge \int_{D\cap\{\sigma \le R\epsilon\}}\zeta_\lambda d\nu> \frac{1}{2}.
\end{equation*}
Taking $R_0=4dR$, we get the desired result.
\end{proof}

We now turn to show that $K\zeta_\lambda$ bifurcates from the Green's function as the vortex-strength $\lambda$ tends to infinity. To begin with, we need some estimates for the Green's function.
\begin{lemma}[\cite{Ta} $Lemma~A\cdot3$]\label{Tadie}
There are constants $C_1$ and $C_2$, depending only on $D$ such that
\begin{equation*}
  \begin{split}
     |K_\xi(x,x_0)| & \le C_1|x-x_0|^{-1} \ \ \ \text{for} \ \ \xi=r,r_0,z,z_0; \\
     |K_{\xi\theta}(x,x_0)| & \le C_2|x-x_0|^{-2} \ \  \ \text{for} \ \ \xi\theta=rr_0,rz_0,zr_0,zz_0.
  \end{split}
\end{equation*}
whenever $(x,x_0)\in D\times D$ and $x\ne x_0$.
\end{lemma}

The next lemma is an analogue of Lemma 5.1 of \cite{BF2}.
\begin{lemma}\label{NB}
Let $M>0$ be fixed, then for any points $x$ and $x_0$ in $D$,
\begin{equation*}
\begin{split}
   \int_{D\cap \{|z|\le M\}}|\nabla_\tau\{K(\tau,x)-K(\tau,x_0)\}|^p d\tau &\le const. |x-x_0|^{2-p}(1+\log\frac{2M+d}{|x-x_0|})\ \ \  \text{if} \ p=1;\\
   \int_{D\cap \{|z|\le M\}}|\nabla_\tau\{K(\tau,x)-K(\tau,x_0)\}|^p d\tau &\le const. |x-x_0|^{2-p}\ \ \text{if}\ 1<p<2,
\end{split}
\end{equation*}
where $d\tau=drdz$ and the constants depend only on $M$, $d$ and $p$.
\end{lemma}

\begin{proof}
  With Lemma $\ref{Tadie}$ in hand, one can repeat the proof in \cite{BF2} without any significant changes, we omit here.
\end{proof}

\begin{proposition}\label{p3}
Let $a(\lambda)$ be any point of $\Omega_\lambda$, let $M>0$ be fixed. Then as $\lambda \to +\infty$,
\begin{equation}\label{309}
  K\zeta_\lambda(\cdot)-{K(\cdot,a(\lambda))} \to 0 \ \  \text{in}\ \  W^{1,p}(D_M),\ \ ~1\le p<2,
\end{equation}
and hence in $L^r(D_M)$, $1\le r<\infty$, where $D_M:=D\cap\{|z|\le M\}$.
\end{proposition}

\begin{proof}
The proof proceeds as in \cite{BF2}. First, recall that
$$\int_{\Omega_\lambda}\zeta_\lambda(x') 2\pi r'dr'dz'=1.$$
$$ K\zeta_\lambda(x)=\int_{\Omega_\lambda}K(x,x')\zeta_\lambda(x') r'dr'dz',\ \ x'=(r',z'),$$
Hence we have
$$ K\zeta_\lambda(x)-{K(x,a(\lambda))}=\int_{\Omega_\lambda}\{K(x,x')-K(x,a(\lambda))\}\zeta_\lambda(x')2\pi r'dr'dz'.$$
Therefore, with $x=(r,z), \nabla_x=(\partial_r,\partial_z),$
\begin{equation}\label{310}
\begin{split}
      ||K\zeta_\lambda(x)-{K(x,a(\lambda))}||&_{W^{1,p}(D_M)}  \\
      \le& \int_{\Omega_\lambda}\zeta_\lambda(x')\{\int_{D}|\nabla_x K(x,x')-\nabla_x K(x,a(\lambda))|^pdrdz\}^{\frac{1}{p}}2\pi r'dr'dz',
\end{split}
\end{equation}
where we used the Minkowski inequality. Notice that $x'$ and $a(\lambda)$  in the inner integral on the right are both in $\Omega_\lambda$, by Proposition $\ref{le18}$, we have
$$|x'-a(\lambda)|\le diam(\Omega_\lambda)\le4d \lambda^{-\frac{1}{4}} .$$
Lemma $\ref{NB}$ now shows that in $\eqref{310}$ the inner integral tends to zero uniformly, and the remaining expression is bounded. Hence we obtain $\eqref{309}$. Convergence to zero in $L^r$ ($ 1 \le r <\infty$) clearly follows from Sobolev embedding. The proof is thus completed.
\end{proof}

\begin{remark}
  Since regularity theory shows that $K(\cdot,a(\lambda))\in W^{1,p}_{loc}(D)$ only for $ p<2 $, the condition $1 \le p <2$ cannot be improved.
\end{remark}

\begin{proposition}\label{p4}
Let $a(\lambda)$ be any point of ~~$\Omega_\lambda$. Then, as $\lambda \to +\infty$,
\begin{equation}\label{311}
  K\zeta_\lambda(\cdot)-{K(\cdot,a(\lambda))} \to 0 \ \  \text{in}\ \  C^{1,\alpha}_{loc}(D\backslash\{(r_*,0)\}),
\end{equation}
where $\alpha$ is any constant in $(0,1)$.
\end{proposition}

In the proof of Proposition $\ref{p4}$, we follow the strategy of \cite{LYY}.
\begin{proof}
By Proposition $\ref{p3}$, we have
\begin{equation}\label{312}
  K\zeta_\lambda(\cdot)-{K(\cdot,a(\lambda))} \to 0 \ \  \text{in}\ \  W^{1,p}_{loc}(D),\
\end{equation}
for $a(\lambda)\in \Omega_\lambda$, where $\ 1\le p<2$. On the other hand, Proposition $\ref{le18}$ yields that
for any $\delta>0$, $$\Omega_\lambda\subseteq B_{\delta}((r_*,0))\ \ \text{for}\ \lambda\ \text{large enough}.$$
Note that,
$$\mathcal{L}(K\zeta_\lambda)=0\ \ \ \text{in}\ \ D \backslash B_{2\delta}((r_*,0)).$$
By the $L^p$-estimate for  elliptic equations, we obtain for any $q>1$
\begin{equation}\label{313}
  ||K\zeta_\lambda||_{{W^{2,q}}(B_\eta(x_0))}\le C_1||K\zeta_\lambda||_{L^q(B_{2\eta}(x_0))}
\end{equation}
for any $B_{2\eta}(x_0)\subseteq D \backslash B_{2\delta}((r_*,0))$. By Lemma $\ref{le1}$, one can easily check that
\begin{equation}\label{314}
  |K(x,a(\lambda))|\le C_2 \ \ \ \forall~ x\in D \backslash B_{2\delta}((r_*,0)).
\end{equation}
Combining $\eqref{312}$, $\eqref{313}$ and $\eqref{314}$, we get
$$||K\zeta_\lambda||_{{W^{2,q}}(B_{\eta}(x_0))}\le C_1||K\zeta_\lambda||_{L^q(B_{2\eta}(x_0))}\le 1+ C_3||K(\cdot,a(\lambda))||_{L^q(B_{2\eta}(x_0))}\le C_4.$$
Therefore, the conclusion follows.
\end{proof}

\begin{proof}[Proof of Theorem \ref{thm1}]
(i) and (ii) follow from  Lemma $\ref{le8}$ and Corollary $\ref{le12}$. (iii) follows from Proposition $\ref{le18}$, Lemmas $\ref{asym1}$ and $\ref{sharp}$. To prove (iv), we refer to \cite{B1}. (v) follows from Propositions $\ref{p3}$ and $\ref{p4}$. Finally, it remains to show that $(\psi_\lambda,\zeta_\lambda)$ satisfies $\eqref{2-2}$. Indeed, this clearly follows from Lemma $\ref{burton}$. Therefore the proof is completed.
\end{proof}

Using this method, one can also get the following results.
\begin{proposition}\label{add}
Let $U=\{(r,\theta,z)\in \mathbb{R}^3~|~0\le r<d\}$ and let $D=U\cap \Pi$. For every $W\ge0$ and all sufficiently large $\lambda$, there exists a weak solution $(\psi_\lambda,\zeta_\lambda)$ of $\eqref{1-7}$ satisfying
\begin{itemize}
\item[(i)]For any $p>1$, $0<\alpha<1$, $\psi_\lambda\in W^{2,p}_{\text{loc}}(D)\cap C^{1,\alpha}(\bar{D})$ and satisfies
\begin{equation*}
  \mathcal{L}\psi_\lambda=\zeta_\lambda\ \ \text{a.e.} \ \text{in} \ D.
\end{equation*}

\item[(ii)] $(\psi_\lambda,\zeta_\lambda)$ is of the form
\begin{equation*}
  \begin{split}
    & \psi_\lambda=K\zeta_\lambda-\frac{W}{2}r^2-\mu_\lambda,\ \ \zeta_\lambda=\lambda I_{\Omega_\lambda}, \\
     & \Omega_\lambda=\{x\in D~|~\psi_\lambda(x)>0 \},\ \ \ \lambda|\Omega_\lambda|=1,
\end{split}
\end{equation*}
for some $\mu_\lambda>0$ depending on $\lambda$.

 \item[(iii)] For any $\alpha\in(0,1)$, there holds $$diam(\Omega_\lambda) \le 4d \lambda^{-\frac{\alpha}{2}}$$
 provided $\lambda$ is large enough.
 Moreover,
 \begin{equation*}
\begin{split}
   \lim_{\lambda\to +\infty} \frac{\log diam(\Omega_\lambda)}{\log (\lambda^{-\frac{1}{2}})} & =1, \\
   \lim_{\lambda \to +\infty}dist_{\mathcal{C}_{d}}(\Omega_\lambda)&=0.
\end{split}
\end{equation*}
\item[(iv)] Let $$\mathbf{v}_\lambda=\frac{1}{r}\Big(-\frac{\partial\psi_\lambda}{\partial z}\mathbf{e}_r+\frac{\partial\psi_\lambda}{\partial r}\mathbf{e}_z\Big),$$ then
\begin{equation*}
\begin{split}
    &\ \ \ \ \ \ \mathbf{v_\lambda}\cdot\mathbf{n}=0 \  \ \text{on}\ \partial U,\\
    &\ \ \ \ \ \ \mathbf{v}_\lambda\to -W\mathbf{e}_z\ \  \text{at}\ \infty, \  \  \text{as}\  \lambda \to
    +\infty,
\end{split}
\end{equation*}
where $\mathbf{n}$ is the unit outward normal of $\partial U$. Moreover, as $r\to 0$,
\begin{equation*}
     \frac{1}{r}\frac{\partial\psi_\lambda}{\partial z}\to 0\ \text{and}~ \ \frac{1}{r}\frac{\partial\psi_\lambda}{\partial r}\  \text{approaches a finite limit}.
\end{equation*}
\item[(v)]Let $a(\lambda)$ be any point of ~~$\Omega_\lambda$. Then, as $\lambda \to +\infty$,
\begin{equation*}
  K\zeta_\lambda(\cdot)-{K(\cdot,a(\lambda))} \to 0 \ \  \text{in}\ \  W^{1,p}(D_M),\ \ ~1\le p<2,
\end{equation*}
and hence in $L^r(D_M)$, $1\le r<\infty$, where $D_M:=D\cap\{|z|\le M\}$. Moreover, for any $\alpha \in (0,1)$,
\begin{equation*}
  K\zeta_\lambda(\cdot)-{K(\cdot,a(\lambda))} \to 0 \ \  \text{in}\ \  C^{1,\alpha}_{loc}(D).
\end{equation*}
\end{itemize}
\end{proposition}

\noindent \textbf{3.2.~Vortex Ring in the Whole Space}

In this subsection we consider the case of the whole space. To this aim we let $D=\Pi=\{(r,z)\in \mathbb{R}^2~|~0<r<+\infty\}$. We note that in this case $K(r,z,r',z')=G(r,z,r',z')$. Let
\begin{equation*}
  G\zeta(r,z):=\int_D G(r,z,r',z')\zeta(r',z')2\pi r'dr'dz'.
\end{equation*}
For fixed $W>0$ and $\lambda>1$, we consider the energy as follows
\[E_\lambda(\zeta)=\frac{1}{2}\int_D{\zeta (G\zeta-Wr^2\log\lambda)}d\nu.\]

\begin{lemma}\label{le331}
For any fixed $W>0$ and $\lambda>1$, there exists $\zeta=\zeta_\lambda \in \mathcal{RC}_\lambda $ such that
\begin{equation*}
 E_\lambda(\zeta)= \max_{\tilde{\zeta} \in \mathcal{WR}_\lambda}E_\lambda(\tilde{\zeta}).
\end{equation*}
If $\zeta_\lambda \neq 0$, then $\zeta_\lambda \in \mathcal{R}_\lambda$ has compact support in $D$; meanwhile,
\begin{equation*}
\zeta_\lambda=\lambda I_{\Omega_\lambda} \ \ a.e.\  \text{in}\  D,
\end{equation*}
where
\begin{equation*}
  \Omega_\lambda=\{(r,z)\in D~|~K\zeta_\lambda(r,z)-\frac{W\log\lambda}{2}r^2>\mu_\lambda \},
\end{equation*}
and the Lagrange multiplier $\mu_\lambda> 0$ is determined by $\zeta_\lambda$.
\end{lemma}

\begin{proof}
  For proofs, we refer to \cite{BB}.
\end{proof}

\begin{remark}
When $\zeta_\lambda$ has compact support in $D$, $K\zeta_\lambda=G\zeta_\lambda \in W^{2,p}_{\text{loc}}(D)\cap C^{1,\alpha}_{\text{loc}}(\bar{D})$ for any $p>1$, $0<\alpha<1$.
\end{remark}

As before, we can obtain the lower bound estimate of energy.

\begin{lemma}\label{le332}
  For any $a\in \mathbb{R}_+$, there exists $C>0$ such that for all $\lambda$ sufficiently large,~we have
\begin{equation*}
  E_\lambda(\zeta_\lambda)\ge (\frac{a}{16\pi^2}-\frac{Wa^2}{2})\log\lambda-C,
\end{equation*}
where the positive number $C$ depends only on $a$, but not on $\lambda$.
\end{lemma}

\begin{corollary}\label{le333}
If $\lambda$ is large enough, then $\zeta_\lambda \in \mathcal{R}_\lambda$.
\end{corollary}

Repeating the discussions in \S3.1, one may obtain the following results.
\begin{proposition}\label{le334}
There exists a constant $R_0>1$ independent of $\lambda$ such that
 \begin{equation*}
   diam(\Omega_\lambda) \le R_0 \lambda^{-\frac{1}{2}}
 \end{equation*}
 provided $\lambda$ is large enough.
Moreover, one has
\begin{equation*}
\begin{split}
   \lim_{\lambda \to +\infty}dist_{\mathcal{C}_{r_*}}(\Omega_\lambda)=0,&\\
   E_\lambda(\zeta_\lambda)   = (\frac{r_*}{16\pi^2}-\frac{Wr_*^2}{2})\log\lambda+O(1),&\\
    \mu_\lambda             = (\frac{r_*}{8\pi^2}-\frac{Wr_*^2}{2})\log\lambda+O(1).&
\end{split}
\end{equation*}
Here $r_*=1/(16\pi^2W)$.
\end{proposition}

The following lemma is a variant of Lemma $\ref{NB}$ in the whole space. For proofs, one can refer to \cite{Ta} and then argue as before, we omit here.
\begin{lemma}\label{Berger2}
Let $V\subseteq D$ be a bounded set, then for any points $x$ and $x_0$ in $V$,
\begin{equation*}
  \int_{V}|\nabla\{K(\tau,x)-K(\tau,x_0)\}|^p d\tau \le const. |x-x_0|^{2-p}(1+\log\frac{diam V}{|x-x_0|})^2, \  d\tau=drdz,
\end{equation*}
where $1\le p<2$ and the constant depends only on $p$ and $V$.
\end{lemma}

Using Lemma $\ref{Berger2}$, it is not hard to obtain the following results

\begin{proposition}\label{forget}
Let $a(\lambda)$ be any point of ~~$\Omega_\lambda$. Then, as $\lambda \to +\infty$,
\begin{equation*}
  K\zeta_\lambda(\cdot)-{K(\cdot,a(\lambda))} \to 0 \ \  \text{in}\ \  W^{1,p}_{\text{loc}}(D),\ \ ~1\le p<2,
\end{equation*}
and hence in $L^r_{\text{loc}}(D)$, $1\le r<\infty$. Moreover,
\begin{equation*}
  K\zeta_\lambda(\cdot)-{K(\cdot,a(\lambda))} \to 0 \ \  \text{in}\ \  C^{1,\alpha}_{loc}(D\backslash\{(r_*,0)\}),
\end{equation*}
where $\alpha$ is any constant in $(0,1)$.
\end{proposition}

\begin{proof}[Proof of Theorem \ref{thm2}]
(i) and (ii) follow from  Lemma $\ref{le331}$ and Corollary $\ref{le333}$. (iii) follows from Proposition $\ref{le334}$. To prove (iv), we refer to \cite{BB}. (v) follows from Proposition $\ref{forget}$. The rest of proof is the same as before, the proof is completed.
\end{proof}

\begin{remark}
Our method is quite different from \cite{FT}. Note that we do not require the impulse of the flow to be a constant. The velocities at infinity of our solutions are determined.
\end{remark}

\section{Vortex Rings Outside a Ball}

In this section, we investigate vortex rings outside a ball. The approach here is a little different from the previous sections.

Let $D=\{(r,z)\in \Pi ~|~ r^2+z^2>d^2\}$ for some $d>0$.

For fixed $W>0$ and $\lambda>1$, we consider the energy as follows
\[E_\lambda(\zeta)=\frac{1}{2}\int_D{\zeta K\zeta}d\nu-\frac{W\log\lambda}{2}\int_{D}r^2\zeta d\nu+\frac{W\log\lambda}{2}\int_{D}\frac{r^2d^3}{(r^2+z^2)^{\frac{3}{2}}}\zeta d\nu.\]
We introduce the function $\Gamma_2$ as follows
\begin{equation*}
  \Gamma_2(t)=t-8\pi^2Wt^2+\frac{8\pi^2Wd^3}{t},\ t\in(0,+\infty).
\end{equation*}
Let $r_*\in[d,+\infty)$ such that $\Gamma_2(r_*)=\max_{t\in[d,+\infty)}\Gamma_2(t)$. It is easy to check that $r_*$ is unique and well-defined. Moreover, $r_*=d$ if $W\ge1/(24\pi^2d)$; $r_*>d$ if $W<1/(24\pi^2d)$.

Let $\mathcal{R}_\lambda$ be defined as in Section 3 and $D_1=\{(r,z)\in D~|~0<r<r_*+1, |z|<2\}$. Define
$$\mathcal{R}_\lambda(D_1)=\{\zeta\in\mathcal{R}_\lambda~|~ \zeta=0 \ \text{a.e. in}\  \Pi\backslash D_1\}.$$

\begin{lemma}\label{le341}
For any fixed $W>0$ and $\lambda>1$, there exists $\zeta=\zeta_\lambda \in \mathcal{R}_\lambda(D_1) $ such that
\begin{equation*}
 E_\lambda(\zeta)= \max_{\tilde{\zeta} \in \mathcal{R}_\lambda(D_1)}E_\lambda(\tilde{\zeta}).
\end{equation*}
For any maximizer $\zeta_\lambda$,~we have $K\zeta_\lambda\in W^{2,p}_{\text{loc}}(D)\cap C^{1,\alpha}_{\text{loc}}(\bar{D})$ for any $p>1$, $0<\alpha<1$. Moreover,
\begin{equation}\label{903}
\zeta_\lambda=\lambda I_{\Omega_\lambda} \ \ a.e.\  \text{in}\  D,
\end{equation}
where
\begin{equation*}
  \Omega_\lambda=\{(r,z)\in D_1~|~K\zeta_\lambda(r,z)-\frac{W\log\lambda}{2}r^2+\frac{W\log\lambda}{2}\frac{r^2d^3}{(r^2+z^2)^{\frac{3}{2}}}>\mu_\lambda \},
\end{equation*}
and the Lagrange multiplier $\mu_\lambda\in \mathbb{R}$ is determined by $\zeta_\lambda$.
\end{lemma}

\begin{proof}
  By Lemma $\ref{le2}$, we deduce that $K(r,z,r',z')\in L^1(D_1\times D_1)$, which further implies that the functional $\int_D \zeta K \zeta d\nu$ is weakly-star continuous in $L^\infty(D_1\times D_1)$. With this in hand, the existence and regularity now follow from Burton \cite{B2}. Moreover, there exists an increasing function $\varphi_\lambda$ such that $\zeta_\lambda=\varphi_\lambda(K\zeta_\lambda)$ almost everywhere in $D$. Since the value of $\zeta_\lambda $ can only be 0 or $\lambda$, we derive that  there exists a $\mu_\lambda \in \mathbb{R} $ such that
  \[\{x\in D~|~\psi_\lambda(x)>0 \} \subseteq {supp}(\zeta_\lambda)\subseteq \{x\in D~|~\psi_\lambda(x)\ge 0 \},\]
  where $$\psi_\lambda=K\zeta_\lambda-\frac{W\log\lambda}{2}r^2+\frac{W\log\lambda}{2}\frac{r^2d^3}{(r^2+z^2)^{\frac{3}{2}}}-\mu_\lambda.$$
  On the level set $\{x\in D~|~\psi_\lambda(x)=0 \}$,~by the property of Sobolev functions, we have $\zeta_\lambda=\mathcal{L}(\psi_\lambda)=0$ a.e. on $\{x\in D~|~\psi_\lambda(x)=0 \}$, from which we obtain $\eqref{903}$. The proof is completed.
\end{proof}

\begin{lemma}\label{le342}
For any $a\in (d,r_*+1)$, there exists $C>0$ such that for all $\lambda$ sufficiently large,~we have
\begin{equation*}
  E_\lambda(\zeta_\lambda)\ge (\frac{a}{16\pi^2}-\frac{Wa^2}{2}+\frac{Wd^3}{2a})\log\lambda-C,
\end{equation*}
where the positive number $C$ depends only on $a$, but not on $\lambda$.
\end{lemma}
\begin{proof}
  Arguing as in the proof of Lemma $\ref{le9}$.
\end{proof}

We estimate the energy of the vortex core. Recall that
\[\psi_\lambda=K\zeta_\lambda-\frac{W\log\lambda}{2}r^2+\frac{W\log\lambda}{2}\frac{r^2d^3}{(r^2+z^2)^{\frac{3}{2}}}-\mu_\lambda.\]~The kinetic energy of the vortex core is defined as follows
\begin{equation*}
  J(\zeta_\lambda)=\frac{1}{2}\int_D \psi_\lambda \zeta_\lambda d\nu.
\end{equation*}

\begin{lemma}\label{le343}
  $J(\zeta_\lambda)\le C$.
\end{lemma}

\begin{proof}
  First, we claim that $\mu_\lambda>-1/2$ when $\lambda$ is large enough. Indeed, if $\mu_\lambda<0$, then
\begin{equation*}
  \{(r,z)\in D_1~|~\frac{W\log\lambda}{2}r^2<|\mu_\lambda|\}\subseteq \Omega_\lambda.
\end{equation*}
Since $\lambda |\Omega_\lambda|=1$, it follows that $|\mu_\lambda|<W\log\lambda/(4\pi\lambda)$, which clearly implies the claim.
Let $u_\lambda=\psi_\lambda-1$. It is not hard to check that $u_{\lambda}^+\in H(D)$ and $u_{\lambda}^+(r,z)=0$ on $D\backslash Q$, where $Q=\{(r,z)\in D~|~r< M,|z|< M\}$ for some $M\in\mathbb{R}_+$ independent of $\lambda$.
Notice that
\begin{equation*}
\begin{split}
   2J(\zeta_\lambda) & =\int_D \psi_\lambda \zeta_\lambda d\nu \\
     & =\int_D (\psi_\lambda-1) \zeta_\lambda d\nu+\int_D \zeta_\lambda d\nu\\
     & \le \int_D u_{\lambda}^+ \zeta_\lambda d\nu+1.
\end{split}
\end{equation*}
Recall that $\mathcal{L}\psi_\lambda=\zeta_\lambda$ in $D$, we can take $u_{\lambda}^+$ as a test function to obtain
\begin{equation*}
\begin{split}
   \int_D \frac{|\nabla u_{\lambda}^+|^2}{r}d\nu &= \int_D u_{\lambda}^+ \zeta_\lambda d\nu\\
                                                 &=\lambda\int_{\Omega_\lambda}u_{\lambda}^+d\nu \\
                                                 &\le \lambda|\Omega_\lambda|^{\frac{1}{2}}(\int_{\Omega_\lambda}|u_{\lambda}^+|^2d\nu)^{\frac{1}{2}} \\
                                                 &\le C_1 \lambda|\Omega_\lambda|^{\frac{1}{2}}(\int_{Q}|u_{\lambda}^+|^2drdz)^{\frac{1}{2}} \\
                                                 &\le C_2 \lambda|\Omega_\lambda|^{\frac{1}{2}}\int_{Q}|\nabla u_{\lambda}^+|drdz   \\
                                                 &\le C_3 \lambda|\Omega_\lambda|(\int_{\Omega_\lambda}\frac{|\nabla u_{\lambda}^+|^2}{r^2}d\nu)^{\frac{1}{2}}.
\end{split}
\end{equation*}
Hence $J(\zeta_\lambda)\le C$, the proof is completed.
\end{proof}

By the definition of $J(\zeta_\lambda)$, we know
\begin{lemma}\label{le344}
As $\lambda \to +\infty$,
$$\mu_\lambda=2E_\lambda(\zeta_\lambda)+W(\log \lambda)\mathcal{I(\zeta_\lambda)}-\frac{W\log\lambda}{2}\int_{D}\frac{r^2d^3}{(r^2+z^2)^{\frac{3}{2}}}\zeta_\lambda d\nu+O(1).$$
\end{lemma}

We now study the asymptotic behaviour of the vortex core. Define
\[A_{\lambda}=\inf\{r_\lambda~|~(r_\lambda,z_\lambda)\in \Omega_\lambda\ \text{for some}\  z_\lambda \in \mathbb{R}\},\ \ B_{\lambda}=\sup\{r_\lambda~|~(r_\lambda,z_\lambda)\in \Omega_\lambda\ \text{for some}\  z_\lambda \in \mathbb{R} \}.  \]
It is easy to know that $(A_{\lambda},Z_{\lambda1}),\  (B_{\lambda},Z_{\lambda2}) \in \bar{\Omega}_{\lambda}$ for some $Z_{\lambda1},\ Z_{\lambda2} \in \mathbb{R}$. Let $g(r,z)=r^2d^3/(r^2+z^2)^{\frac{3}{2}}$.
\begin{lemma}\label{le345}
  $\lim_{\lambda \to +\infty}A_\lambda=r_*$.
\end{lemma}

\begin{proof}
Let $\alpha \in (0,1)$ and $\epsilon=\lambda^{-\frac{1}{2}}$. Arguing as in the proof of Lemma $\ref{le13}$, we have for any point $(r_\lambda,z_\lambda)\in \bar{\Omega}_\lambda$
\begin{equation*}
  K\zeta_\lambda(r_\lambda,z_\lambda)\le \frac{r_\lambda\log\lambda}{8\pi^2}\int_{B_{C_0\epsilon^\alpha}((r_\lambda,z_\lambda))}\zeta_\lambda d\nu+C\lambda^{-\frac{\alpha}{2}}\log \lambda+C+\frac{C}{8\pi^2} \sinh^{-1}(\lambda^{\frac{\alpha}{2}}).
\end{equation*}
On the other hand, for any $a\in(d,r_*+1)$,
\begin{equation*}
\begin{split}
   K\zeta_\lambda&(r_\lambda,z_\lambda)-\frac{W\log\lambda}{2}(r_\lambda)^2+\frac{W\log\lambda}{2}g(r_\lambda,z_\lambda) \\
    & \ge (\frac{a}{8\pi^2}-Wa^2+\frac{Wd^3}{a})\log\lambda+W(\log \lambda)\mathcal{I(\zeta_\lambda)}-\frac{W\log\lambda}{2}\int_{D}g(r,z)\zeta_\lambda d\nu-C.
\end{split}
\end{equation*}
Thus
\begin{equation*}
\begin{split}
   (\frac{a}{8\pi^2}&-Wa^2+\frac{Wd^3}{a})\log\lambda+W(\log \lambda)\mathcal{I(\zeta_\lambda)}-\frac{W\log\lambda}{2}\int_{D}g(r,z)\zeta_\lambda d\nu\\
    \le& \frac{r_\lambda\log\lambda}{8\pi^2}\int_{B_{C_0\epsilon^\alpha}((r_\lambda,z_\lambda))}\zeta_\lambda d\nu-\frac{W\log\lambda}{2}(r_\lambda)^2+\frac{W\log\lambda}{2}g(r_\lambda,z_\lambda) +C\lambda^{-\frac{\alpha}{2}}\log \lambda\\
       &+\frac{C}{8\pi^2} \sinh^{-1}(\lambda^{\frac{\alpha}{2}})+C.
\end{split}
\end{equation*}
Dividing both sides by $\log \lambda/(8\pi^2)$, we obtain
\begin{equation}\label{346}
\begin{split}
   \Gamma&_2(a)+4\pi^2W[2\mathcal{I(\zeta_\lambda)}-(r_\lambda)^2]+4\pi^2W[g(r_\lambda,0)-\int_D g(r,z)\zeta_\lambda d\nu]\\
              &\le\Gamma_2(a)+4\pi^2W[2\mathcal{I(\zeta_\lambda)}-(r_\lambda)^2]+4\pi^2W[2g(r_\lambda,0)-\int_D g(r,z)\zeta_\lambda d\nu-g(r_\lambda,z_\lambda)]\\
             &\le  r_\lambda \int_{B_{C_0\epsilon^\alpha}((r_\lambda,z_\lambda))}\zeta_\lambda d\nu-8\pi^2W(r_\lambda)^2+\frac{8\pi^2Wd^3}{r_\lambda}+\frac{C\sinh^{-1}(\lambda^{\frac{\alpha}{2}})}{\log\lambda}+C\lambda^{-\frac{\alpha}{2}}+\frac{C}{\log\lambda}   \\
             &= \Gamma_2(r_\lambda)+\frac{C\sinh^{-1}(\lambda^{\frac{\alpha}{2}})}{\log\lambda}+C\lambda^{-\frac{\alpha}{2}}+\frac{C}{\log\lambda}.
\end{split}
\end{equation}
Notice that
\begin{equation*}
\begin{split}
   &2\mathcal{I(\zeta_\lambda)} =\int_{D}r^2\zeta_\lambda d\nu \ge (A_\lambda)^2, \\
     &g(A_{\lambda},0)-\int_D g(r,z)\zeta_\lambda d\nu\ge 0.
\end{split}
\end{equation*}
Taking $r_\lambda=A_\lambda$ and $z_\lambda=Z_{\lambda1}$, $\eqref{346}$ leads to
\begin{equation*}
  \Gamma_2(a)\le \Gamma_2(A_\lambda)+\frac{C\sinh^{-1}(\lambda^{\frac{\alpha}{2}})}{\log\lambda}+C\lambda^{-\frac{\alpha}{2}}+\frac{C}{\log\lambda}.
\end{equation*}
Now letting $\lambda$ tend to $+\infty$, we deduce that
\begin{equation*}
  \Gamma_2(a)\le \liminf_{\lambda\to +\infty}\Gamma_2(A_\lambda)+\frac{C\alpha}{2}.
\end{equation*}
Hence we get the desired result by letting $a\to r_*$ and $\alpha \to 0$.
\end{proof}

Arguing as in the proofs of Lemma $\ref{le14}$ and Lemma $\ref{le15}$, we obtain
\begin{lemma}\label{le347}
$\lim_{\lambda \to +\infty}\mathcal{I(\zeta_\lambda)}=r_*^2/2$ and $\lim_{\lambda \to +\infty}\int_D g(r,z)\zeta_\lambda d\nu=d^3/r_*$.
Moreover, for any $ \eta >0$, there holds
$$\lim_{\lambda \to +\infty} \int_{D\cap\{{r\ge r_*+\eta}\}}\zeta_\lambda d\nu=0.$$
\end{lemma}

We now further show that the vortex core will shrink to the circle $\mathcal{C}_{r_*}$.

\begin{lemma}\label{le348}
  $\lim_{\lambda \to +\infty}dist_{\mathcal{C}_{r_*}}(\Omega_\lambda)=0$.
\end{lemma}

\begin{proof}
We first prove $\lim_{\lambda \to +\infty}B_\lambda=r_*$. Suppose $r_*>d$, from $\eqref{346}$ we conclude that for any $\alpha>0$, $\epsilon:=\lambda^{-\frac{1}{2}}$,
\begin{equation*}
\begin{split}
   \Gamma_2(r_*)+4\pi^2Wr_*^2+4\pi^2W\liminf_{\lambda \to +\infty}(B_\lambda)^2-\frac{4\pi^2Wd^3}{r_*}&-4\pi^2Wd^3\limsup_{\lambda \to +\infty}\frac{1}{r_\lambda}     \\
     & \le C\liminf_{\lambda \to +\infty}\int_{B_{C_0\epsilon^\alpha}((B_\lambda,Z_{\lambda2}))}\zeta_\lambda d\nu+\frac{C\alpha}{2}.
\end{split}
\end{equation*}
Since
\begin{equation*}
  \begin{split}
       & \liminf_{\lambda \to +\infty}(B_\lambda)^2\ge r_*^2, \\
       & \limsup_{\lambda \to +\infty}{1}/{r_\lambda}\le 1/r_*,
  \end{split}
\end{equation*}
we get
\begin{equation*}
  r_*\le C \liminf_{\lambda \to +\infty}\int_{B_{C_0\epsilon^\alpha}((B_\lambda,Z_{\lambda2}))}\zeta_\lambda d\nu+\frac{C\alpha}{2}.
\end{equation*}
Taking $\alpha$ so small such that $r_*-C\alpha/2>0$, the desired result follows from Lemma $\ref{le347}$. The proof for the case $r_*=d$ is similar.

 We now turn to prove $\lim_{\lambda \to +\infty}dist_{\mathcal{C}_{r_*}}(\Omega_\lambda)=0$. Recall that Lemma $\ref{le345}$, it suffices to prove that for any $(r_\lambda,z_\lambda)\in \Omega_\lambda$, we have $z_\lambda\to 0$ as $\lambda \to +\infty$. In fact, by $\eqref{346}$, for any $0<\alpha<1$,
\[\frac{d^3}{r_*}\le \liminf_{\lambda \to +\infty}g(r_\lambda,z_\lambda)+\frac{C\alpha}{2},\]
it follows that $z_\lambda\to 0$. The proof is completed.
\end{proof}

Note that Lemma $\ref{Berger2}$ also holds in this case. With these results in hand, it is now not difficult to use the previous methods to obtain the following

\begin{proposition}\label{le348}
For any $\alpha\in(0,1)$, there holds
\begin{equation*}
\begin{split}
   &diam(\Omega_\lambda) \le C_0 \lambda^{-\frac{\alpha}{2}}, \\
   &dist(\Omega_\lambda, \partial D_1)>0,
\end{split}
\end{equation*}
 provided $\lambda$ is large enough, where $C_0>0$ is independent of $\lambda$ and $\alpha$.
Moreover,
\begin{equation*}
\begin{split}
   \lim_{\lambda\to +\infty} \frac{\log diam(\Omega_\lambda)}{\log (\lambda^{-\frac{1}{2}})} & =1, \\
   \lim_{\lambda \to +\infty}dist_{\mathcal{C}_{r_*}}(\Omega_\lambda)&=0.
\end{split}
\end{equation*}
As $\lambda \to +\infty$,
\begin{equation*}
  \mathbf{v}_\lambda=\frac{1}{r}\Big(-\frac{\partial\psi_\lambda}{\partial z}\mathbf{e}_r+\frac{\partial\psi_\lambda}{\partial r}\mathbf{e}_z\Big)\to -W\log\lambda~\mathbf{e}_z\ \ \text{at} \ \infty.
\end{equation*}
Let $a(\lambda)$ be any point of ~~$\Omega_\lambda$, then, as $\lambda \to +\infty$,
\begin{equation*}
  K\zeta_\lambda(\cdot)-{K(\cdot,a(\lambda))} \to 0 \ \  \text{in}\ \  W^{1,p}_{\text{loc}}(D),\ \ ~1\le p<2,
\end{equation*}
and hence in $L^r_{\text{loc}}(D)$, $1\le r<\infty$. Moreover, for any $\alpha\in(0,1)$,
\begin{equation*}
  K\zeta_\lambda(\cdot)-{K(\cdot,a(\lambda))} \to 0 \ \  \text{in}\ \  C^{1,\alpha}_{loc}(D\backslash\{(r_*,0)\}).
\end{equation*}
\end{proposition}

Arguing as in the proofs of Lemma $\ref{asym1}$ and Lemma $\ref{sharp}$, we also have
\begin{lemma}\label{asym2}
Suppose $ W< 1/(24\pi^2d)$, then as $\lambda\to +\infty$,
\begin{equation*}
\begin{split}
    E_\lambda(\zeta_\lambda) & = (\frac{r_*}{16\pi^2}-\frac{Wr_*^2}{2}+\frac{Wd^3}{2r_*})\log\lambda+O(1), \\
        \mu_\lambda          & = (\frac{r_*}{8\pi^2}-\frac{Wr_*^2}{2}+\frac{Wd^3}{2r_*})\log\lambda+O(1).
\end{split}
\end{equation*}
Moreover, there exists a constant $R_0>1$ independent of $\lambda$ such that
$$diam(\Omega_\lambda)\le R_0\lambda^{-\frac{1}{2}}$$
provided $\lambda$ is large enough.
\end{lemma}

Having made all the preparation we are now ready to give proof of Theorem \ref{thm3}.

\begin{proof}[Proof of Theorem \ref{thm3}]
Note that
\begin{equation*}
  \begin{split}
      \mathcal{L}\psi_\lambda&=0\ \ \text{in}\  D\backslash \bar{\Omega}_\lambda, \\
      \psi_\lambda&\le 0\ \ \text{on}\  \partial D \cup \partial \bar{\Omega}_\lambda, \\
     \psi_\lambda &\le 0 \ \ \text{at}\  \infty .
  \end{split}
\end{equation*}
By the maximum principle, we conclude that $\psi_\lambda\le 0$ in $D\backslash \bar{\Omega}_\lambda$. Hence $$\Omega_\lambda=\{(r,z)\in D~|~\psi_\lambda>0 \}.$$ The rest of proof is the same as before, the proof is completed.
\end{proof}

\begin{remark}
  In the previous cases, one might expect to construct a family of vortex rings in this way. But we prefer to use the previous method, because those solutions contain more information which may be used for further study.
\end{remark}

\section{Vortex Rings in Bounded Domains}
Now we turn to study vortex rings in bounded domains. Since the method is the same as before, we will briefly describe some key steps here and omit other details.

Let $D=\{(r,z)\in \mathbb R^2~|~r^2+z^2<b^2\}$ or $(0,b)\times(-c,c)\ \text{for some}\  b,c\in \mathbb{R}_+$.

Due to the presence of the wall, it is natural to require that the fluid does not cross the boundary.
Hence we consider the kinetic energy of the flow as follows
\[E(\zeta)=\frac{1}{2}\int_D{\zeta K\zeta}d\nu.\]
We adopt the class of admissible functions $\mathcal{R}_\lambda$ as follows
\begin{equation*}
\mathcal{R}_\lambda=\{\zeta\in L^\infty(D)~|~ \zeta= \lambda I_A \ \text{for some measurable subset}\  A\subseteq D , \int_D \zeta d\nu=1\},
\end{equation*}
where $\lambda$ is a positive number which we assume throughout the sequel that $\lambda > |D|^{-1}$. Note that $\mathcal{R}_\lambda$ is not empty.

As before, we have the following result firstly.
\begin{lemma}\label{le3}
There exists $\zeta=\zeta_\lambda \in \mathcal{R}_\lambda $ such that
\begin{equation}\label{304}
 E(\zeta)= \max_{\tilde{\zeta} \in \mathcal{R}_\lambda}E(\tilde{\zeta}).
\end{equation}
For any maximizer $\zeta_\lambda$,~we have $K\zeta_\lambda\in W^{2,p}_{\text{loc}}(D)\cap C^{1,\alpha}(\bar{D})$ for any $p>1$, $0<\alpha<1$. Moreover,
\begin{equation}\label{305}
\zeta_\lambda=\lambda I_{\Omega_\lambda} \ \ a.e.\  \text{in}\  D,
\end{equation}
where
\begin{equation*}
  \Omega_\lambda=\{x\in D~|~K\zeta_\lambda(x)>\mu_\lambda \},
\end{equation*}
and the Lagrange multiplier $\mu_\lambda>0$ is determined by $\zeta_\lambda$.
\end{lemma}

Next we require some results on Steiner symmetrization from Appendix I of Fraenkel and Berger \cite{BF1}. Let ${\zeta}^*$ be the Steiner symmetrization of $\zeta$ with respect to the line $z=0$ in $D$. A similar argument as in \cite{B1} yields that

\begin{lemma}\label{le4}
  Let $\zeta \in L^2(D)$ be non-negative. Then $K\zeta \geq 0$ and $E(\zeta^*)\geq E(\zeta)$. Further if $\zeta^*=\zeta$ then $(K\zeta)^*=K\zeta$.
\end{lemma}
Note that if $\zeta_\lambda$ is a maximizer then so is $(\zeta_\lambda)^*$. Henceforth we shall assume $\zeta_\lambda=(\zeta_\lambda)^*$. Hence $K\zeta_\lambda=(K\zeta_\lambda)^*$ is a symmetrically deceasing function on each line parallel to the $z$ axis. The following lemma gives the lower bound estimate of the energy.

\begin{lemma}\label{le5}
  For any  $a\in(0,b)$, there exists $C>0$ such that for all $\lambda$ sufficiently large,~we have
\begin{equation*}
  E(\zeta_\lambda)\ge \frac{a}{16\pi^2}\log\lambda-C,
\end{equation*}
where the positive number $C$ depends only on $a$, but not on $\lambda$.
\end{lemma}
Then we estimate the energy of the vortex core. Let $\psi_\lambda=K\zeta_\lambda-\mu_\lambda$. The kinetic energy of the vortex core is defined as follows
\begin{equation*}
  J(\zeta_\lambda)=\frac{1}{2}\int_D\frac{{|\nabla (\psi_\lambda^+)|^2}}{r^2}d\nu\ ,
\end{equation*}
where ${\psi_\lambda^+} = \max \{\psi_\lambda,0\}$.

\begin{lemma}\label{le6}
  $J(\zeta_\lambda)\le C$.
\end{lemma}

We are now ready to estimate the Lagrange multiplier $\mu_\lambda$. Notice that
$$2E(\zeta_\lambda)=\int_{D} \zeta_\lambda (K\zeta_\lambda) d\nu=2\int_{D} \zeta_\lambda (K\zeta_\lambda-\mu_\lambda) d\nu+\mu_\lambda=2J(\zeta_\lambda)+\mu_\lambda,$$
which, combined with Lemma $\ref{le6}$, deduces the following result
\begin{lemma}\label{le7}
  $\mu_\lambda=2E(\zeta_\lambda)+O(1)$, as $\lambda \to +\infty$.
\end{lemma}

We can now study the properties of the vortex core as before.
\begin{proposition}\label{p1}
 For any $\alpha\in(0,1)$, there holds
 \begin{equation*}
   diam(\Omega_\lambda) \le 4b \lambda^{-\frac{\alpha}{2}}
 \end{equation*}
 provided $\lambda$ is large enough.
 Moreover, one has
\begin{equation*}
\begin{split}
   \lim_{\lambda\to +\infty} \frac{\log diam(\Omega_\lambda)}{\log (\lambda^{-\frac{1}{2}})} & =1,\\
\lim_{\lambda \to +\infty}dist_{\mathcal{C}_{b}}(\Omega_\lambda) & =0.
\end{split}
\end{equation*}
\end{proposition}

To show that $K\zeta_\lambda$ bifurcates from the Green's function as the vortex-strength $\lambda$ tends to infinity, we need the following estimate.

\begin{lemma}[{\cite{BF2}, $Lemma~5.1$}]\label{Berger}
For any points $x$ and $x_0$ in $D$,
\begin{equation*}
  \int_{D}|\nabla\{K(\tau,x)-K(\tau,x_0)\}|^p d\tau \le const |x-x_0|^{2-p}(1+\log\frac{diam D}{|x-x_0|})^2, \  d\tau=drdz,
\end{equation*}
where $1\le p<2$ and the constant depends only on $D$ and $p$.
\end{lemma}

Using Lemma $\ref{Berger}$, one can argue as in section 3 to obtain
\begin{proposition}\label{p7}
Let $a(\lambda)$ be any point of ~~$\Omega_\lambda$. Then, as $\lambda \to \infty$,
\begin{equation}\label{320}
  K\zeta_\lambda(\cdot)-{K(\cdot,a(\lambda))} \to 0 \ \  \text{in}\ \  W^{1,p}_{0}(D),\ \ ~1\le p<2,
\end{equation}
and hence in $L^r(D)$, $1\le r<\infty$. Moreover, for any $\alpha\in(0,1)$,
\begin{equation}\label{321}
  K\zeta_\lambda(\cdot)-{K(\cdot,a(\lambda))} \to 0 \ \  \text{in}\ \  C^{1,\alpha}_{loc}(D).
\end{equation}
\end{proposition}

\begin{proof}[Proof of Theorem \ref{thm4}]
 (i) and (ii) follow from  Lemma $\ref{le3}$. (iii) follows from Proposition $\ref{p1}$. To prove (iv), we refer to Lemma 9 of \cite{B2}. (v) follows from Proposition $\ref{p7}$. Finally, it follows from Lemma $\ref{burton}$ that $(\psi_\lambda,\zeta_\lambda)$ satisfies $\eqref{2-2}$. The proof is completed.
\end{proof}

{\bf Acknowledgments.}
{ This work was supported by NNSF of China (No.11771469) and Chinese Academy of Sciences (
  No.QYZDJ-SSW-SYS021).
}

\phantom{s}
 \thispagestyle{empty}

\end{document}